\documentclass[review,onefignum,onetabnum,nohypdvips]{siamart171218}

\usepackage{mathrsfs,amsfonts,amsmath,amssymb,bm,mathtools}
\usepackage{multirow,multicol}
\usepackage{algorithmic,algorithm}
\usepackage{booktabs,lscape,url,tabularx,colortbl}
\usepackage[draft]{changes}
\usepackage{setspace}
\usepackage{subfigure}%
\usepackage{tikz}
\usetikzlibrary{shapes.geometric, arrows}
\usepackage{ulem}
\usepackage[titletoc]{appendix}
\usepackage{etoolbox, xstring, mfirstuc, textcase}

\newcommand{\ba}{\bm{a}}

\newcommand{\bh}{\bm{h}}

\newcommand{\bu}{\bm{u}}
\newcommand{\by}{\bm{y}}
\newcommand{\bz}{\bm{z}}
\newcommand{\bx}{\bm{x}}
\newcommand{\bk}{\bm{k}}
\newcommand{\bmm}{\bm{m}}
\newcommand{\bj}{\bm{j}}
\newcommand{\bv}{\bm{v}}
\newcommand{\bp}{\bm{p}}

\newcommand{\bw}{\bm{w}}

\newcommand{\DbV}{\Delta \bm{V}}

\newcommand{\bbeta}{\bm{\beta}}
\newcommand{\blam}{\bm{\lambda}}
\newcommand{\bLam}{\bm{\Lambda}}

\newcommand{\bell}{\bm{\ell}}

\newcommand{\bA}{\bm{A}}

\newcommand{\bP}{\bm{P}}

\newcommand{\bL}{\bm{L}}
\newcommand{\bG}{\bm{G}}

\newcommand{\bbT}{\mathbb{T}}
\newcommand{\bbR}{\mathbb{R}}

\newcommand{\bbC}{\mathbb{C}}
\newcommand{\bbZ}{\mathbb{Z}}
\newcommand{\bbQ}{\mathbb{Q}}

\newcommand{\ellinf}{\ell^{\infty}}

\newcommand{\calP}{\mathcal{P}}

\newcommand{\calM}{\mathcal{M}}

\newcommand{\calU}{\mathcal{U}}

\newcommand{\calT}{\mathcal{T}}

\newcommand{\fraku}{\mathfrak{u}}
\newcommand{\hF}{\hat{F}}

\newcommand{\QP}{\mbox{QP}}

\newcommand\tbbint{{-\mkern -16mu\int}}

\newcommand\dbbint{{-\mkern -19mu\int}}

\newcommand\bbint{
	{\mathchoice{\dbbint}{\tbbint}{\tbbint}{\tbbint}}
}

\allowdisplaybreaks[4]
\DeclareMathOperator*{\argmin}{\mathrm{argmin}}

\newtheorem{thm}{Theorem}[section]
\newtheorem{defy}[thm]{Definition}
\newtheorem{remark}[thm]{Remark}
\newtheorem{assumption}[thm]{Assumption}

\newtheorem{example}[thm]{Example}

\title{On the approximation of quasiperiodic functions with Diophantine frequencies by periodic functions
	\thanks{Submitted to ...
				\funding{
		}
}
}
\author{Kai Jiang\thanks{
		Hunan Key Laboratory for Computation and Simulation in Science and Engineering,
		Key Laboratory of Intelligent Computing and Information Processing of Ministry of Education, School of Mathematics and Computational Science, Xiangtan University, Xiangtan, Hunan, China, 411105.
		(\email{kaijiang@xtu.edu.cn}, \email{shifengli@smail.xtu.edu.cn}).}
	\and ShiFeng Li\footnotemark[2]
	\and Pingwen Zhang\thanks{
		School of Mathematics and Statistics, Wuhan University, Wuhan, 430072, School of Mathematical Sciences, Peking University, Beijing, 100871, China.
		(\email{pzhang@pku.edu.cn}). } 
}

\usepackage{array}
\newcolumntype{I}{!{\vrule width 1.0pt}}
\newlength\savedwidth

\newlength\savewidth

\begin{document}
	\maketitle
	
\begin{abstract}
We present an analysis of the approximation error for a $d$-dimensional quasiperiodic function $f$ with Diophantine frequencies, approximated by a periodic function with the fundamental domain $[0,L_1)\times [0,L_2)\times \cdots \times[0,L_d)$. 
When $f$ has a certain regularity, its global behavior can be described by a finite number of Fourier components and has a polynomial decay at infinity. 
The dominant part of periodic approximation error is bounded by $O(\max_{1\leq j \leq d} L_j^{-s_j})$, where $L_j$ belongs to the best simultaneous approximation sequence and $s_j$ is the number of different irrational elements in $j$-th dimension component of Fourier frequencies, respectively.
Meanwhile, we discuss the approximation rate.
Finally, these analytical results are verified by some examples. 
\end{abstract}	
	


	\section{Introduction}
	\label{sec:intr}




Quasiperiodic systems, as a natural extension of periodic systems, are widely found in nature, materials science and physical systems, such as many-body problems, incommensurate structures, quasicrystals, polycrystalline materials, and quasiperiodic quantum systems \cite{J.F.periodic2010,Aubry1980Effects, shechtman1984metallic,Avila2009spectrum,Avila2015Sharp,Verbin2012Observation,Kraus2013Four-dimensional}.
We consider a $d$-dimensional quasiperiodic function $f(\bx)$, and aim to approximate it using a periodic function $f_{p}(\bx)$ in a finite domain. 
This subject is basic and important in the field of approximation theory. Meanwhile, this is also the core idea of periodic approximation method (PAM), which is widely used to study quasiperiodic systems \cite{A.M.2008,Damanik2014isospectral,jiang2014numerical}.
Therefore, studying the approximation error of PAM not only expands the intension of the approximation theory, but also establishes a basic theory for the application of PAM.
However, it is surprising that there is still a lack of rigorous and systematic theoretical analysis on this approximation problem.
In this work, we analyze the approximation error of multi-dimensional quasiperiodic functions with Diophantine frequencies when approximated by periodic functions.



Assume that $f(\bx)$ has the Fourier series \cite{levitan1982almost}
\begin{align}
f(\bx)=\sum_{\blam\in \bLam} a_{\blam}
e^{i 2\pi \blam\cdot\bx},~\bx\in \bbR^{d},
\label{eq:quasiperiodic-n}
\end{align}
where 
\begin{align}
a_{\blam}
=\lim_{T\rightarrow +\infty}
\frac{1}{(2T)^d}\int_{[-T,T]^d}f(\bx)e^{-i2\pi\blam\cdot \bx}\,d\bx
=\bbint f(\bx)e^{-i2\pi\blam\cdot \bx}\,d\bx
\label{eq:QPFouriercoff}
\end{align}
is the Fourier coefficient and $\bLam=\{\blam: \blam = \bP\bk,~ \bk\in\bbZ^n \}\subset \bbR^d$ is the Fourier exponent set (or called Fourier frequency set).
$\bP=(\bp_{1},\bp_{2},\cdots,\bp_{n})\in\bbR^{d\times n}$,
$n\geq d$, is the projection matrix
where $\bp_{1},\bp_{2},\cdots,\bp_{n}$ are $\mathbb{Q}$-linearly independent.

	
Let $F$ be the $n$-dimensional periodic parent function of $f$ such that $f(\bx)=F(\bP^{T}\bx)$. The convergence of the Fourier series of its parent function can be determined by the convergence of the Fourier series of quasiperiodic function, and vice versa\,\cite{Jiang2022Numerical}. Therefore, for the quasiperiodic function $f$ given in \eqref{eq:quasiperiodic-n}, $F$ can be expanded as
\begin{align*}
F(\bz)
=\sum_{\bk\in \bbZ^n} \hF_{\bk}
e^{i 2\pi \bk\cdot\bz},~\bz\in \bbT^{n}=\bbR^n/\bbZ^n,
\end{align*}
with Fourier coefficient
\begin{align}
\hF_{\bk}=\frac{1}{\vert \bbT^n\vert}
\int_{\bbT^n}e^{-i2\pi \bk\cdot\bz}F(\bz)\,d\bz.
\label{eq:ParentFouriercoff}
\end{align}
More properties of parent function $F$ can refer to \cite{Jiang2022Numerical}.
	
Based on the Birkhoff's ergodic theorem \cite{pitt1942generalizations}, we have the following useful result
\begin{theorem}[\cite{Jiang2022Numerical}]
\label{thm:object}
For a given quasiperiodic function
\begin{align*}
f(\bx)=F(\bp_1\cdot \bx,\cdots,\bp_n\cdot \bx), ~~~\bx\in\bbR^d,
\end{align*}
where $F(\bz)$ is its parent function defined on $\bbT^n$,
$\bp_1,\cdots,\bp_n$ are $\bbQ$-linearly independent, we have
\begin{align*}
a_{\blam}=\hF_{\bk},
\end{align*}
where $\blam=\bP\bk$, $\bk\in\bbZ^n$. $a_{\blam}$ and $\hF_{\bk}$ are defined in \eqref{eq:QPFouriercoff} and \eqref{eq:ParentFouriercoff}, respectively.
\end{theorem}
According to \Cref{thm:object}, Fourier coefficients $\hF_{\bk}$ of $F$ and Fourier coefficients $a_{\blam}$ of $f$ have the same decay behavior. Consequently, if the parent function $F$ has certain regularity, it is reasonable to define the set of fundamental Fourier exponents that globally describe the behavior of $f$.
For a given positive integer $N\in \bbZ^+$, denote
\begin{align*}
K_N^n=\{\bk=(k_j)_{j=1}^n \in\bbZ^n: \, -N \leq  k_j < N \},
\end{align*}
and the fundamental Fourier exponents set of $f$ can be defined as
\begin{align*}
\bLam^d_{N}=\{\blam = \bP\bk: \bk\in K_N^n \}\subset \bLam.
\end{align*}
Obviously, the order of the set $\bLam^d_{N}$ is $\# (\bLam^d_{N})=D=(2N)^n$.
Let $\QP(\bbR^d)$ represent the space of all $d$-dimensional quasiperiodic functions.
From $\bLam^d_{N}$, we can obtain a finite dimensional linear subspace of $\QP(\bbR^d)$
\begin{align*}
S_N=\mbox{span} 
\{ e^{i2\pi\blam\cdot \bx}, \bx\in\bbR^d, \blam\in\bLam^d_{N}\}.
\end{align*}
Denote the projection operator $\calP_N: \QP(\bbR^d) \mapsto S_N$. Then we can split the quasiperiodic function $f$ into two parts
\begin{align*}
f(\bx) =\sum_{\blam_\ell \in \bLam_N^d } a_{\ell} e^{i 2\pi \blam_\ell \cdot \bx} +\sum_{\blam_\ell \in \bLam \backslash \bLam_N^d} a_{\ell} e^{i 2\pi \blam_\ell \cdot \bx}
=\calP_N f+(f-\calP_N f).
\end{align*}
From the viewpoint of $\bx$-space, a periodic function is used to approximate $\calP_N f$.
Concretely, for some vectors $\bx=(x_j)^d_{j=1},~\by=(y_j)^d_{j=1}$ and $\bz=(z_j)^d_{j=1}$ with $z_j\neq 0$, we define Hadamard product $\bx \circ\by = (x_j y_j)^d_{j=1}$ and $\bx/\bz=(x_j/z_j)^d_{j=1}$, $\vert \bz\vert=z_1z_2\cdots z_d$.
Then, for a given positive integer vector $\bL=(L_j)^d_{j=1}$, we rewrite $\calP_N f$ as
\begin{align*}
\calP_N f(\bx)
=\sum_{\ell=1}^{D} a_{\ell}
e^{i 2\pi \bv_\ell \cdot \bx/\bL},
\end{align*}
where $\bv_{\ell}=\bL\circ\blam_{\ell}$ with $\blam_{\ell}\in \bLam_N^d$.
Using a periodic function 
\begin{align}
f_{p}(\bx)
=\sum_{\ell=1}^{D} b_{\ell}
e^{i 2\pi \bh_{\ell}\cdot \bx/\bL},~~\bx\in \bbT^d= \bbR^d/(\bL\circ\bbZ^d),
\label{eq:periodic-trunction-n}
\end{align}
where $\bh_{\ell}\in\bbZ^d$ and $b_\ell$ is the Fourier coefficient, to approximate $\calP_N f$ in $\Omega=[0,L_1)\times [0,L_2)\times \cdots \times [0,L_d)$. 
	
	
In related work, Gomez \textit{et al.} applied the periodic approximation method to recover frequencies and amplitudes of a one-dimensional quasiperiodic function from regular sampling data\,\cite{G.M.S.1.2010}. Correspondingly, a special error analysis was also provided for their considered one-dimensional cases\,\cite{gomez2010collocation}.
The analysis in high-dimensional cases meet many difficulties compared to one-dimensional ones. The main challenge is that irrational frequencies in Fourier exponents may exist across different dimensions.
In this paper, we are devoted to giving a theoretical analysis of the periodic approximation problem for arbitrary dimensional quasiperiodic function.


From the viewpoint of reciprocal space,
periodic approximation problem involves using the integer vector $\bh_{\ell}$ to approximate the irrational vector $\bv_{\ell}$, which is related to the Diophantine approximation theory. For any vector $\bx=(x_j)_{j=1}^d\in\bbR^d$, let $[\bx]=([x_j])_{j=1}^d$ denote the integer vector whose element $[x_j]$ is the distance between $x_j$ and its nearest integer and $\Vert \bx \Vert_{\ellinf}=\max\limits_{1\leq j\leq d} \vert x_j\vert $. 
In \Cref{subsec:Dio-matrix}, we will demonstrate that $f_{p}$ present a good approximation to $f$ when $\bh_{\ell}=[\bv_{\ell}]$. Correspondingly, we can obtain the Diophantine inequality
\begin{align*}
\Vert \bh_{\ell}-\bv_{\ell}\Vert_{\ellinf}=\max\limits_{1\leq j\leq d}\vert v_{\ell,j}-h_{\ell,j}\vert <1/2.
\end{align*} 
Clearly, the approximation analysis of Fourier frequencies is directly related to Diophantine theory. The following Dirichlet's theorem by Dirichlet is a fundamental result in this field.
	
\begin{theorem}[Dirichlet's theorem on simultaneous approximation \cite{W.M.S.1980}]
\label{thm:Dirichletapp}
Suppose that $\alpha_1,\cdots, \alpha_s$ are $s$ real numbers. Then there are infinitely integer points $(q, p_1,\cdots,p_s)$ with $q\neq 0$ such that
\begin{align*}
\max_{1\leq j\leq s} \vert \alpha_j q- p_j  \vert <C_sq^{-1/s},
\end{align*}
where $C_s=s/(s+1)$.
\end{theorem}
Denote $Y^d_D=\{\bv_1,~\bv_2,\cdots, \bv_D\}$. 
According to \Cref{thm:Dirichletapp}, there exists an increasing sequence $\{q_{j,1},~q_{j,2},\cdots\}$ such that
\begin{align*}
E_{L_j}= \sum_{\ell=1}^D\vert v_{\ell,j}-h_{\ell,j}\vert
=\sum_{\ell=1}^D \vert L_j \lambda_{\ell,j} - [L_j \lambda_{\ell,j}]\vert
\leq DC_{s_j} L_j^{-1/s_j},
\end{align*}
where $L_j\in \{q_{j,1},~q_{j,2},\cdots\}$ and $s_j$ the number of different irrational elements in $j$-th dimension of $Y^d_D$, $j=1,2,\cdots, d$.

\begin{defy}
\label{def:bestapprox}
For $j=1,2,\cdots, d$, the $j$ column increasing subsequence $\calT_j(Y^d_{D})=\{t_{j,1},~t_{j,2},\cdots\}\subset \{q_{j,1},~q_{j,2},\cdots\}$ is the best simultaneous approximation sequence of $Y^d_D$ by taking $t_{j,1}=q_{j,1}$ and $t_{j,k}=\argmin\limits_{E_{t_{j,k}}<E_{t_{j,k-1}}}
\{q_{j,\ell}\}^{\infty}_{\ell=\ell_k}$ with $t_{j,k-1}=q_{j,\ell_k}$.
\end{defy}

Assume that the first $\zeta$ Fourier exponents of $Y^d_{D}$ belong to $\bbQ^d$, the rest belongs to $\bbR^d \backslash \bbQ^d$. In fact, $\bv_j=\bh_j~(j=1,2,\cdots,\zeta)$.
Denote $Y^{d}_{\zeta}=\{\bv_1,\bv_2,\cdots,\bv_{\zeta}\}$.
Without loss of generality,
we always have $Y^{d}_{\zeta}\neq \emptyset$.
Otherwise, we can obtain a new vector $\bv_1 \in \bbQ_*^d=\bbQ^d\backslash \{\bf{0}\}$
through dividing $v_{j\ell}$ by $v_{1\ell}$, where $j=1,2,\cdots,D,~\ell=1,2,\cdots,d$ and $v_{1\ell}\neq 0$.

To analyze the approximation error, we introduce the definition of Diophantine number.

\begin{defy}[\cite{Milnor2006Dynamics}]
\label{def:diophantine number}
A real number $\alpha$ is said to be Diophantine number if for any $\tau>0$ and there exists a constant $C_a>0$ such that  
\begin{align*}
	\Big \vert \alpha-\frac{p}{q}\Big \vert\geq \frac{C_a}{q^{2+\tau}}
\end{align*}
for every rational number $p/q$.
\end{defy}
	
	From the definition of Diophantine number, we give the Diophantine condition on the fundamental Fourier exponents set $\bLam_N^d$.	
	\begin{assumption}
		\label{ass:DiopCond}
		Assume that irrational elements $\lambda_{\ell, j}$ in the Fourier exponents  
		$ \blam_{\ell}=\bP\bk\in Y^d_{D}\setminus Y^d_{\zeta}$ are Diophantine numbers.
		In particular, $\blam_{\ell}$ satisfy Diophantine condition when $p=0,~ q=\|\bk\|_{\ellinf}$ and
		\begin{align*}
			\vert \lambda_{\ell, j} \vert
			>\frac{C_a}{\|\bk \|_{\ellinf}^{2+\tau}},~~\tau>0.
		\end{align*} 	
	\end{assumption}
	%

\begin{remark}
An irrational number is either a Diophantine number or a Liouville number. The set of Liouville numbers has a Hausdorff dimension of zero \cite{Milnor2006Dynamics}. Specifically, all algebraic numbers in $\bbR\backslash\bbQ$ are Diophantine numbers \cite{R.1955}.
\end{remark}

\begin{remark}
Since $\bv = \bL\circ\blam$ and all rational elements $\lambda_j$ in $\blam$ can be transformed into integers by an appropriate choice of $\bL$, our analysis focuses on distinguishing between integers and irrational numbers in $\bv$.
\end{remark}


{\bf Main results:}
The approximation error between the quasiperiodic function and the approximated periodic function is measured by the infinity norm $\|f_{p}-f\|_{\infty} =\sup_{\bx\in \Omega}|f_{p}(\bx)-f(\bx)|$.
According to the triangle inequality, the approximation error becomes
\begin{align}
\| f_{p}-f\|_{\infty} \leq
\|f_{p}-\calP_N f\|_{\infty}
+\|\calP_N f-f\|_{\infty}.
\label{eq:error-bound-n}
\end{align}
The right terms of inequality \eqref{eq:error-bound-n} are the rational approximation error and the truncation error, respectively. The truncation error is related to the regularity of the quasiperiodic function $f$. If $f$ is $\alpha$-order derivative, the truncation error can be bounded by $O(N^{\kappa-\alpha})$ where $\alpha>\kappa>d/2$.
The rational approximation error is estimated as $O(\max_{1\leq j \leq d} L_j^{-s_j})$ where $s_j$ is the number of different irrational elements in $j$-th dimension of $Y_D^d$ and $L_j\in \calT_j(Y^d_{D})$.
The detailed analysis of the main results will be presented in \Cref{sec:proof}.

\section{Analysis}
\label{sec:proof}
	
Let $\ba=(a_{1},a_{2},\cdots,a_{n})^{T}\in\bbR^n$ and $\bA=(a_{ij})\in\bbR^{d\times n}$, denote that $\|\ba\|=\sum_{j=1}^{n}|a_{j}|$ and 
$\|\bA\|_{1}=\max\limits_{1\leq j\leq n}\sum_{i=1}^d|a_{ij}|$.
Denote $A\lesssim B$ as an estimate of the form $A<c B$ for a positive constant $c$. For any positive integer $\alpha$, the space $H^\alpha_{QP}(\bbR^d)$ comprises all quasiperiodic functions with partial derivatives order $\alpha \geq 1$ with respect to the inner product $(\cdot, \cdot)_{\alpha}$
	\begin{align*}
		(f_1, f_2)_{\alpha}
		=\bbint f_1 \bar{f}_2 d\bx 
		+ \sum_{\Vert \bmm\Vert=\alpha} \bbint \partial^{\bmm}_{\bx} f_1\cdot \overline{\partial^{\bmm}_{\bx} f_2} d\bx,
	\end{align*}
	and $\partial^{\bmm}_{\bx}=\partial^{m_1}_{x_1}\cdots\partial^{m_d}_{x_d}$.
	Then, we can define	the norm
	$\Vert f \Vert_{\alpha}^2=\sum_{\bk\in \bbZ^n}(1+\Vert \blam_{\bk}\Vert^2)^{\alpha}\vert  a_{\blam_{\bk}} \vert^2,$
	and semi-norm
	$\vert f \vert_{\alpha}^2=\sum_{\bk\in \bbZ^n}\Vert \blam_{\bk}\Vert^{2\alpha}\vert a_{\blam_{\bk}} \vert^2.$
	

Assume that the quasiperiodic function $f\in H_{QP}^\alpha(\bbR^d)$, the estimate of $\|\calP_N f-f\|_{\infty}$ has been given in \cite{Jiang2022Numerical},
\textit{i.e.,}
\begin{align*}
\|\calP_N f-f\|_{\infty} \lesssim N^{\kappa-\alpha}\vert f \vert_\alpha,
\end{align*}
where $\alpha>\kappa>d/2$. The truncation error becomes negligible when $f$ exhibits sufficient regularity. Hence, the periodic approximation error is mainly dominated by the rational approximation error.

Denote $b_{max}=\max\limits_{1 \leq \ell\leq D}\{\vert b_{\ell}\vert \}$. Next, we estimate the rational approximation error
$\|f_{p}-\calP_N f\|_{\infty}$. 
According to the definition of $f_{p}$ and $\calP_N f$, we have
\begin{align}
\| f_{p} -\calP_N f\|_{\infty}
&=\sup\limits_{\bx\in \Omega}|\calP_N f(\bx)-f_{p}(\bx)|\notag\\
&\leq \Big|\sum_{\ell=1}^{D} a_{\ell}
e^{i 2\pi \bv_\ell\cdot \bx/\bL}
-\sum_{\ell=1}^{D} b_{\ell} e^{i 2\pi \bh_\ell\cdot\bx/\bL}\Big|\notag\\
&= \Big|\sum_{\ell=1}^{D} (a_{\ell}-b_{\ell})
e^{i 2\pi \bv_\ell\cdot\bx/\bL}
-\sum_{\ell=1}^{D} b_{\ell}
(e^{i 2\pi \bh_\ell\cdot\bx/\bL}
-e^{i 2\pi \bv_\ell\cdot\bx/\bL})\Big|\notag\\
&\leq \sum_{\ell=1}^{D} \vert a_{\ell} - b_{\ell}\vert\cdot
\Big|e^{i 2\pi \bv_\ell\cdot\bx/\bL}\Big|
+\sum_{\ell=1}^{D} |b_{\ell}|\cdot
\Big|e^{i 2\pi \bv_\ell\cdot\bx/\bL}\Big|\cdot
\Big|e^{i 2\pi (\bh_\ell-\bv_\ell)\cdot\bx/\bL}-1\Big|\notag\\
&= \sum_{\ell=1}^{D} \vert a_{\ell}-b_{\ell} \vert
+\sum_{\ell=1}^{D} |b_{\ell}|\cdot
\Big|2\sin[\pi (\bh_\ell-\bv_\ell)\cdot\bx/\bL]\Big|\notag\\
&< \sum_{\ell=1}^{D} \vert a_{\ell}-b_{\ell}\vert
+2\pi b_{max} \sum_{\ell=1}^{D} \sum_{j=1}^{d} \vert h_{\ell,j}- v_{\ell,j}\vert.
\label{ineq:tf-f-1}
\end{align}
In fact, $\|\bh_\ell-\bv_\ell\|$ can be arbitrary small (see \Cref{subsec:Dio-matrix}).
Hence, the last inequality in \eqref{ineq:tf-f-1} is reasonable when $\vert h_{\ell,j}-v_{\ell,j}\vert \leq 1/2d,~j=1,2,\cdots,d$.

Denote the quasiperiodic and periodic Fourier coefficient vectors
\begin{align*}
	\by=(a_{1},a_{2},\cdots,a_{D})^{T},~~
	\by_p=(b_{1},b_{2},\cdots,b_{D})^{T},
\end{align*}
the difference is $\Delta \by=\by-\by_p$,
and
\begin{align*}
	\DbV=(\bh_1-\bv_1,\bh_2-\bv_2,\cdots,\bh_{D}-\bv_{D})
	\in \bbR^{d\times D}.
\end{align*}
Define $\Vert \DbV \Vert_{e}= \sum_{\ell=1}^{D} \sum_{j=1}^{d} \vert h_{\ell,j}- v_{\ell,j}\vert$.
Then, the inequality \eqref{ineq:tf-f-1} can be reduced to
\begin{align}
\|f_{p}-\calP_N f\|_{\infty}< \|\Delta \by\|+2\pi b_{max}\Vert \DbV\Vert_{e},
\label{ineq:tf-f-2}
\end{align}
where $\|\DbV\|_e$ is the Diophantine approximation error and $\|\Delta \by\|$ is the frequency approximation error.
Then we will estimate $\|\Delta \by\|$ and $\|\DbV\|_e$ in \Cref{subsec:approximation-error-y} and \Cref{subsec:Dio-matrix}, respectively.

\subsection{Error estimation $\Vert \Delta \bf y \Vert$}
	\label{subsec:approximation-error-y}
	
In this subsection, we will estimate the upper bound of $\Vert \Delta \by\Vert$ with the help of discrete Fourier transform (DFT). The windowed DFT with $G_j$ discretization nodes in $j$-th dimension is
\begin{align*}	
F_{f,\bL,\bG}^{\eta}(\bbeta)=
\frac{1}{\vert \bG\vert }\sum_{\bj\in K^d_{\bG} }H^{\eta }_{\bG}(\bj)
f(\bj \circ \bL/\bG)e^{-i2\pi\bbeta\cdot\bj/\bG},
\end{align*}
where $\bG = (G_{\ell})_{\ell=1}^d$,
$K^d_{\bG}=\{\bj=(j_\ell)^d_{\ell=1}\in \bbZ^d: -G_{\ell}/2\leq j_\ell\leq G_{\ell}/2-1 \}$, $\bbeta\in\bbZ^d$ and
\begin{equation}
H^{\eta}_{\bG}(\bj)=\left\{
\begin{aligned}
&\mathop{\prod}_{\ell=1}^{d}\frac{\eta !}{(2\eta -1)!!}
\left(1-\cos\frac{2\pi
j_{\ell}}{G_{\ell}}\right)^{\eta },
&&\bj\in K^d_{\bG},\\
&0,&&\mbox{otherwise}.\nonumber
\end{aligned}\right.
\end{equation}
Note that one can always choose $\bG$ such that $G_\ell\gg L_\ell$. Concretely, we require that $\calP_N f$ and $f_{p}$ are equal through the DFT with $\eta$-order Hanning windowed function
	\begin{align}
		F_{\calP_N f,\bL,\bG}^{\eta }(\bh_s)=F_{f_{p},\bL,\bG}^{\eta }(\bh_s), ~~s=1,2,\cdots,D.
		\label{eq:procedure2}
	\end{align}
This is equivalent to
\begin{align*}	
\frac{1}{\vert \bG\vert }\sum_{\bj\in K^d_{\bG}}H^{\eta }_{\bG}(\bj)
\calP_N f(\bj \circ\bL/\bG)e^{-i2\pi\bh_s\cdot\bj/\bG}
=\frac{1}{\vert \bG\vert}\sum_{\bj\in K^d_{\bG}}H^{\eta }_{\bG}(\bj)
f_{p}(\bj \circ\bL/\bG)e^{-i2\pi\bh_s\cdot\bj/\bG},
\end{align*}
where $\bh_s\in X^d_D=\{\bh_1,~\bh_2,\cdots, \bh_D\}
	\subseteq K^d_{\bG}$.
	The matrix form of \eqref{eq:procedure2} is
	\begin{align}
		M \by=M_p \by_p,
		\label{eq:matrixForm}
	\end{align}
where
\begin{align}
&M=(u_{st})\in\bbC^{D\times D},~u_{st}
=F^{\eta}_{e^{i2\pi \bv_t\cdot\bx/\bL},\bL,\bG}(\bh_s)
=\frac{1}{\vert \bG\vert}\sum_{\bj\in K^d_{\bG}}H^{\eta }_{\bG}(\bj)
e^{i2\pi(\bv_t-\bh_s)\cdot\bj/\bG},\label{dey:u}\\
&M_p=(u^p_{st})\in\bbC^{D\times D},~u^p_{st}
=F^{\eta}_{e^{i2\pi \bh_t\cdot\bx/\bL},\bL,\bG}(\bh_s)
=\frac{1}{\vert \bG\vert}\sum_{\bj\in K^d_{\bG}}H^{\eta }_{\bG}(\bj)
e^{i2\pi(\bh_t-\bh_s)\cdot\bj/\bG}.\notag
\end{align}
From \eqref{eq:matrixForm}, we can obtain the Fourier coefficient vector $\by_p$ if $\by$ is known, and vice versa.
	
\Cref{subsec:boundinvM} demonstrates that $M$ is invertible if the Fourier exponent $\bv_t$ satisfies \Cref{ass:DiopCond}, and $\bL,~\bG$ satisfy \Cref{assum:L-N}. As a result, by applying the norm property to the linear system \eqref{eq:matrixForm}, we can obtain
\begin{align*}
\Vert \Delta \by\Vert \leq  b_{max} \Vert M^{-1}\Vert_1\Vert \Vert M_p-M\Vert_e .
\end{align*}
	
In the following, we give the upper bound estimates of $\|M_p-M\|_e$ and $\|M^{-1}\|_1$ in \Cref{sub:error-M-p-n} and \Cref{subsec:boundinvM}, respectively.
	
\subsubsection{The bound of $\Vert M_p-M\Vert_e$}
	\label{sub:error-M-p-n}
	
	
	The difference between $u_{st}$ in $M$ and $u^p_{st}$ in $M_p$ can be estimated as
\begin{align*}
|u_{st}-u^p_{st}|
&=\Big|\frac{1}{\vert \bG\vert}\sum_{\bj\in K^d_{\bG} }H^{\eta }_{\bG}(\bj)e^{i2\pi(\bv_t-\bh_s)\cdot\bj/\bG}
-\frac{1}{\vert \bG\vert}\sum_{\bj\in K^d_{\bG} }H^{\eta }_{\bG}(\bj) e^{i2\pi(\bh_t-\bh_s)\cdot\bj/\bG}\Big|\\
&\leq \frac{1}{\vert \bG\vert}\sum_{\bj\in K^d_{\bG} }H^{\eta }_{\bG}(\bj)\cdot \Big|e^{i2\pi (\bh_t-\bh_s)\cdot\bj/\bG}\Big|\cdot
\Big|e^{i2\pi(\bv_t-\bh_t)\cdot\bj/\bG}-1\Big|\\
&= \frac{1}{\vert \bG\vert}\sum_{\bj\in K^d_{\bG} }H^{\eta }_{\bG}(\bj)\cdot
\Big|2\sin[\pi(\bv_t-\bh_t)\cdot\bj/\bG]\Big|\\
&\leq 2\pi \|\bv_t-\bh_t\|,
\end{align*}
where the last inequality is reasonable with $\vert h_{t,j}-v_{t,j}\vert \leq 1/2d,~j=1,2,\cdots,d$.
Consequently, we have
\begin{align*}
\|M_p-M\|_e=\sum_{t=1}^{D}
\sum_{s=1}^{D}|u^p_{st}-u_{st}|
< 2\pi D \|\DbV\|_e.
\end{align*}
Obviously, the Diophantine approximation error $\|\DbV\|_e$ controls the bound of $\|M_p-M\|_e$.
	
	
\subsubsection{The bound of $\Vert M^{-1}\Vert_1$}
\label{subsec:boundinvM}

\begin{figure}[H]
\begin{center}	
\tikzstyle{startstop} = [rectangle, rounded corners, minimum width=1cm, minimum height=1.6cm, text centered, text width=2cm, draw=black]
\tikzstyle{io} = [rectangle, rounded  corners,minimum width=0.1cm,minimum height=1cm,text centered, text width=3.8cm, draw=black]
\tikzstyle{ioo} = [rectangle, rounded  corners,minimum width=0.3cm,minimum height=1.6cm,text width=2.6cm,text centered,draw=black]
\tikzstyle{process1} = [rectangle, rounded corners, minimum width=0.01cm, minimum height=1.6cm, text centered, text width=4.8cm, draw=black]
\tikzstyle{process2} = [rectangle, rounded  corners, minimum width=0.3cm,minimum height=1.6cm,text width=2.6cm, text centered,draw=black]
\tikzstyle{decision} = [rectangle, rounded  corners, minimum
width=0.04cm,minimum height=0.1cm,text width=1.74cm, text centered,draw=white]
\tikzstyle{arrow} = [thick=50cm,->,>=stealth]
\begin{tikzpicture}[node distance=1.4cm]
minimum	\node (start) [startstop] {Coefficient matrix $M$ \eqref{dey:u}};
\node (input1) [io,right of=start,xshift=2.4cm] 
{ \small{\begin{align*}
M^{-1}=\begin{pmatrix}
I_{\zeta} & -M_{12}U^{-1}\\
0 & U^{-1}
\end{pmatrix}
\end{align*} (\Cref{M-nonsingular})}};
\node (process1) [process1,right of=input1,xshift=1.8cm]at(6,0) {\small{\begin{align*}
\|U^{-1}\|_1
\leq \frac{\|\calU^{-1}\|_1}{1-\|\calU^{-1}\|_1\cdot
\|U-\calU\|_1}
\end{align*} (\Cref{thm:bound-M-n})}};
\node (process2) [process2,right of=process1,xshift=1cm] at (4.8,-2.8){The bound of $\|U-\calU\|_1$ (\Cref{lemma:bound3-n})};
\node (out1) [ioo,right of=process1,xshift=1cm]at (8,-2.8) {The bound of $\|\calU^{-1}\|_1$ (\Cref{lemma:bound1-n})};
\node (decision) [decision,right of=start,xshift=1cm]at (7,-3) {};
\draw [arrow] (start) -- (input1);
\draw [arrow] (input1) -- (process1);
\draw [arrow] (process1.south) to node {} (out1.north);
\draw [arrow] (process1.south) to node [right=-0.6cm]{} (process2.north);
\end{tikzpicture}
\end{center}
\caption{An overview of the upper bound proof of $\|M^{-1}\|_1$.\label{fig:scftiteration}}
\end{figure}
	
This subsection proves that $\Vert M^{-1}\Vert_1$ is bounded and we give its upper bound.
For the purpose of error analysis, we introduce some required notations.
Let $\bbZ^d_*=\bbZ^d\backslash \{\bf{0}\}$ for $d\in\bbZ^+$.
Denote that $I_{0}(\bell)$ and $I_{in}(\bv)$ are index sets of 
zero and integer entries of $\bell\in \bbR^d$, respectively,
\textit{i.e.,}
\begin{align*}
I_{0}(\bell) &=\{j: \ell_j=0, ~1\leq j\leq d \},
\\
I_{in}(\bell) &=\{j: \ell_j\in\bbZ, ~1\leq j\leq d \}.
\end{align*}
For $0\leq r\leq d $, denote
\begin{align}
J_r=\{\bell\in\bbZ^d: \#I_{0}(\bell)=r \}
\subset \bbZ^d.
\label{def:J_r}
\end{align}
Obviously, $\cup^d_{r=0} J_r=\bbZ^d$. 
	

\textbf{Continuous normalized windowed Fourier transform}

The continuous normalized windowed Fourier transform (NWFT) with Hanning window function of order $\eta\in\bbZ^+$ is
\begin{align}	
\phi^{\eta }_{f,\bL}(\bw)=
\frac{1}{|\Omega|}\int_{\Omega}H^{\eta }_{\bL}(\bx)f(\bx)
e^{-i2\pi\bw\cdot\bx}\,d\bx,
\label{eq:NWFT-n}
\end{align}
where $\bw=(w_{j})^{d}_{j=1}\in\bbR^d$ and
\begin{equation}
H^{\eta }_{\bL}(\bx)=\left\{
\begin{aligned}
&\mathop{\prod}_{j=1}^{d}\frac{\eta !}{(2\eta -1)!!}
\left(1-\cos\frac{2\pi x_{j}}{L_j}\right)^{\eta },
&&\bx=(x_{j})^d_{j=1}\in\Omega,\\
&0,&&\mbox{otherwise}.\nonumber
\end{aligned}\right.
\end{equation}
Then $0\leq H^{\eta }_{\bL}(\bx)\leq (2\eta+1)^d$.
The Hanning window function has the normalization property
\begin{align}	
\frac{1}{|\Omega|}\int_{\Omega}H^{\eta }_{\bL}(\bx)\,d\bx=1.
\label{eq:H-int}
\end{align}

From equations \eqref{eq:NWFT-n} and \eqref{eq:H-int},
for a vector $\bv=(v_j)^{d}_{j=1}\in\bbR^d$, we obtain
\begin{align*}
\phi^{\eta }_{e^{i2\pi v_\ell x_\ell},L_\ell}(w_\ell)=1, ~~\ell \in I_{0}(\bv-\bw).
\end{align*}
For a given $L_{\ell}\in\bbZ^+$, if $\vert (v_{\ell}-w_{\ell})L_\ell+j_1\vert > 0$ with $-\eta\leq j_1\leq \eta$, we have
\begin{align}
\phi^{\eta }_{e^{i2\pi\bv\cdot\bx},L_{\ell}}(\bw)
=\mathop{\prod}^{d}
_{\substack{\ell=1\\ \ell\not\in I_{0}(\bv-\bw)}}
\frac{(-1)^{\eta }(\eta !)^{2}[e^{i2\pi(v_{\ell}-w_{\ell})L_{\ell}}-1]}
{i2\pi\Pi^{\eta }_{j_1=-\eta }[(v_{\ell}-w_{\ell})L_{\ell}+j_1]}.
\label{eq:NWFT-complex-exponential-n}
\end{align}
We can also give the coefficient matrix 
$\calM=(\fraku_{st})\in\mathbb{C}^{D\times D}$ in the NWFT
where
\begin{align*}
\fraku_{st}
=\phi^{\eta}_{e^{i2\pi \bv_t\cdot\bx/\bL},\bL}
(\bh_s/\bL).
\end{align*}

The relation between DFT and NWFT is given
in \Cref{lemma:relation-DFT-NWFT-n}.
	
\begin{lemma}[Relation between the DFT and the NWFT]
\label{lemma:relation-DFT-NWFT-n}
When $\eta \geq 1$, we have
\begin{align}	
F^{\eta }_{f,\bL,\bG}(\bk)
=\phi^{\eta }_{f,\bL}(\bk/\bL)
+\sum_{ \bell\in\bbZ_*^d}
\phi^{\eta }_{f,\bL}
\Big(\frac{\bk+\bell  \circ \bG}{\bL}\Big).
\label{eq:relation-DFT-NWFT}
\end{align}

\end{lemma}
The proof of \Cref{lemma:relation-DFT-NWFT-n} is similar to the one-dimensional case presented in \cite{B.1988}.

\textbf{Rewriting the form of $M$}
	
From the definition of $\phi$, we know that
\begin{align}
\fraku_{st}=\begin{cases}
	1,~s=t,\\
	0,~s\neq t,
\end{cases}
~~\mbox{for}~~1\leq s\leq D,~1\leq t\leq \zeta.
\label{eq:calusttau}
\end{align}
According to the properties of $Y^d_{D}$,
we rewrite $M$ and $\calM$ as block matrices
\begin{align}
M=\begin{pmatrix}
	M_{11} & M_{12}\\
	M_{21} & U
\end{pmatrix},~~
\calM=\begin{pmatrix}
	I_{\zeta} & \calM_{12}\\
	0 & \calU
\end{pmatrix},
\label{dey:ucalu}
\end{align}
where $M_{11}\in\bbC^{\zeta\times \zeta}$ and
$U\in\bbC^{(D-\zeta)\times (D-\zeta)}$.
	

From the relation between DFT and NWFT,
we have the following proposition.
\begin{proposition}
\label{pro:M11M21}
Under \Cref{ass:DiopCond},
$\bL$ and $\bG$ satisfy
\begin{align}
\frac{L_{j}C_a}{(2N)^{2+\tau}}-\frac{1}{2}-\eta>0,~~
G_{j}-2L_{j}\|\bP\|_1N - \frac{1}{2}>\eta,~~j=1,2,\cdots,d.
\label{con:pro-NL}
\end{align}
Then, we have $M_{11}=I_{\zeta}$,
~$M_{21}=\bm{0}$.
\end{proposition}

\begin{proof}
For $ \forall \bell\in\bbZ_*^d$, if inequalities
\begin{align}
	\label{ineq:con-Mss}
	\begin{cases}
		|v_{s,j}-h_{t,j}+j_1|>0, ~~~~~\quad\quad j\notin I_{0}(\bv_s-\bh_t)\\
		|v_{s,j}-h_{t,j}-\ell_j G_j+j_1|>0,~~j\notin I_{0}(\bell)
	\end{cases}
\end{align}
hold for $\vert j_1 \vert<\eta $.
From \Cref{lemma:relation-DFT-NWFT-n} and
$v_{s,j}-h_{s,j}-\ell_j G_j=-\ell_j G_j \in\bbZ_*$ with $1\leq s\leq \zeta$,
we can obtain
\begin{align*}
\vert u_{ss}- \fraku_{ss}\vert
&=\vert F_{e^{i2\pi\bv_s\cdot\bx/\bL},\bL,\bG}^{\eta }(\bh_s)-\phi_{e^{i2\pi\bv_s\cdot\bx/\bL},\bL}^{\eta }
(\bh_s/\bL)\vert\\
&\leq \sum_{\bell\in\bbZ_*^d}
\Big\vert\phi^{\eta }_{e^{i2\pi\bv_s\cdot\bx/\bL},\bL}
\Big(\frac{\bh_s+\bell \circ \bG}{\bL}\Big)\Big\vert\\
&=\sum_{r=0}^{d-1}
\sum_{\bell\in J_{r}}
\Bigg\vert\mathop{\prod}^{d}_{j=1,j\notin I_{0}(\bell)}
\frac{(-1)^{\eta }(\eta !)^{2}
[e^{i2\pi(v_{s,j}-h_{s,j}-\ell_j G_j)}-1]}
{i2\pi\Pi^{\eta }_{j_1=-\eta }
(v_{s,j}-h_{s,j}-\ell_j G_j+j_1)}\Bigg\vert
=0.
\end{align*}

Similarly, 
based on inequalities \eqref{ineq:con-Mss},
since $v_{t,j}-h_{s,j}\in\bbZ_*$
and $v_{t,j}-h_{s,j}-\ell_j G_j\in\bbZ_*$ with $1\leq t\leq \zeta$, we have
\begin{align*}
\vert u_{st}- \fraku_{st}\vert
&\leq \sum_{r=0}^{d-1}
\sum_{\bell\in J_{r}}
\Bigg\{
\bigg(\mathop{\prod}_{j\in I_{0}(\bell)\setminus I_{0}(\bv_t-\bv_s)}
\frac{(\eta !)^{2}
	\vert e^{i2\pi(v_{t,j}-h_{s,j})}-1\vert}
{2\pi \Pi^{\eta }_{j_1=-\eta }|v_{t,j}-h_{s,j}+j_1|}\bigg)\\
&~~~\cdot
\bigg(\mathop{\prod}^{d}_{j=1,j\notin I_{0}(\bell)}
\frac{(\eta !)^{2}\vert e^{i2\pi(v_{t,j}-h_{s,j}-\ell_j G_j)}-1\vert}
{2\pi\Pi^{\eta }_{j_1=-\eta }
	\vert v_{t,j}-h_{s,j}-\ell_j G_j+j_1\vert}\bigg)\Bigg\}\\
&=0.
\end{align*}
Therefore, $u_{st}= \fraku_{st}$ where $1\leq s,t\leq \zeta$.
Combining with the equation \eqref{eq:calusttau}, it follows that
\begin{align*}
	u_{st}=\begin{cases}
		1,~s=t,\\
		0,~s\neq t,
	\end{cases}
	~~\mbox{for}~~1\leq s\leq D,~1\leq t\leq \zeta.
\end{align*}
This means that $M_{11}=I_{\zeta}$ and $M_{21}=\bm{0}$.

Next, we show that inequalities \eqref{ineq:con-Mss} are true when $\bL$ and $\bG$ satisfy conditions \eqref{con:pro-NL}.
		
(i) We prove that the inequality $|v_{s,j}-h_{t,j}+j_1|>0$ holds when $j\notin
I_{0}(\bv_s-\bh_t)$ and $\vert j_1\vert \leq\eta$.
For $j\notin I_{0}(\bv_s-\bv_t)$, due to \Cref{ass:DiopCond} and $\vert j_1\vert\leq \eta$, we have
\begin{align}
|v_{s,j}-h_{t,j}+j_1|
\geq \min_{j\notin I_{0}(\bv_s-\bv_t)}
(\vert v_{s,j}-h_{t,j} \vert-\vert j_1 \vert)
>\frac{L_j C_a}{(2N)^{2+\tau}}-\frac{1}{2}-\eta>0.
\label{bound:A0vshs}
\end{align}
For $j\in I_{0}(\bv_s-\bv_t)$, then
$\vert v_{s,j}-h_{t,j}+j_1\vert =\vert v_{t,j}-h_{t,j}+j_1\vert$.
From the Diophantine inequality
$\vert v_{t,j}-h_{t,j}\vert< 1/2$,
we have $|v_{t,j}-h_{t,j}+j_1|>0$ when $v_{t,j}\neq h_{t,j}$.
Therefore, $|v_{s,j}-h_{t,j}+j_1|>0$ is true
for $j\notin I_{0}(\bv_s-\bh_t)$.
		
(ii) We prove that
$|v_{t,j}-h_{s,j}-\ell_j G_j+j_1|>0$ is true when $j\notin I_{0}(\bell)$.
For $j\notin I_{0}(\bv_s-\bv_t)$, since $\vert j_1 \vert\leq \eta$
and $j\notin I_{0}(\bell)$, then
\begin{align*}
|v_{s,j}-h_{t,j}-\ell_j G_j+j_1|
\geq |\ell_j|G_j-|v_{s,j}-h_{t,j}|-|j_1|.
\end{align*}
Since
\begin{align*}
\vert v_{s,j}-h_{t,j}\vert 
\leq \vert v_{s,j}-v_{t,j}\vert+ \vert v_{t,j}-h_{t,j}\vert
<2L_j\Vert \bP\Vert_1 N+1/2,
\end{align*}
we can obtain
\begin{align*}
|v_{s,j}-h_{t,j}-\ell_j G_j+j_1|
\geq |\ell_j|G_j-2L_j\Vert \bP\Vert_1 N-1/2-\eta.
\end{align*}
Moreover, we have
\begin{align*}
	\vert v_{s,j}-h_{t,j}-\ell_j G_j \vert>0,
\end{align*}
that is $j\notin I_{0}(\bv_s-\bh_t-\bell \circ\bG)$.
		
For $j\in I_{0}(\bv_s-\bv_t)$ and $j\notin I_{0}(\bell)$,
we obtain
\begin{align*}
	|v_{s,j}-h_{s,j}-\ell_j G_j+j_1|
	\geq G_j-\|\bv_{s}-\bh_{s}\|_{\ellinf}-\eta
	> G_j-\frac{1}{2}-\eta>0,
\end{align*}
and this also means that $j\notin I_{0}(\bv_s-\bh_s-\bell \circ \bG)$.

Therefore, $|v_{t,j}-h_{s,j}-\ell_j G_j+j_1|>0$ holds for $j\notin I_{0}(\bell)$.
\end{proof}

\begin{remark}
Similarly to \Cref{pro:M11M21},
it is easy to prove that $M_p$ is the identity matrix when conditions \eqref{con:pro-NL} are ture.
\end{remark}

Applying \Cref{pro:M11M21}, $M$ becomes
\begin{align*}
M=\begin{pmatrix}
I_{\zeta} & M_{12}\\
0 & U
\end{pmatrix}.
\end{align*} 

\textbf{Analyze the bound of $\Vert M^{-1}\Vert_1$}

\Cref{M-nonsingular} will show that $U$ is nonsingular,
then
\begin{align*}
M^{-1}=\begin{pmatrix}
I_{\zeta} & -M_{12}U^{-1}\\
0 & U^{-1}
\end{pmatrix}.
\end{align*}
Moreover, we can obtain
\begin{align*}
\|M^{-1}\|_1\leq \max \{1,~(1+\|M_{12}\|_1)\|U^{-1}\|_1\}.
\end{align*}


The upper bound of $\Vert M^{-1}\Vert_1$ can be obtained by estimating bounds of $\|U^{-1}\|_1$ and $\|M_{12}\|_1$.

\textbf{Subproblem 1: The bound of $\|U^{-1}\|_1$}

\label{sub:error-M-n}
Before giving the bound of $\|U^{-1}\|_1$,
we introduce some necessary lemmas and symbols.

\begin{lemma}
(Chapter 5.8 in \cite{H.J.1986})
Assume that $E=E_1 + E_2$.
If $E_{1}$ is invertible with $\|E_{1}^{-1}\|\cdot \|E_{2}\|<1$,
then $E$ is invertible and
\begin{align}
\|E^{-1}\| \leq \frac{\|E_{1}^{-1}\|}{1-\|E_{1}^{-1}\|\cdot \|E_{2}\|}.
\end{align}
\label{lemma:invertible-matrix-bound-n}
\end{lemma}
	
	Set $E_1=\calU,~E_2=U-\calU$ in \Cref{lemma:invertible-matrix-bound-n}.
	\Cref{M-nonsingular} provides the sufficient condition such that $\calU$ is nonsingular and $\|\calU^{-1}\|_1\cdot \|U-\calU\|_1<1$.
	The proof of the upper bound of $\|U^{-1}\|_1$ is splitted into two parts: upper bounds of $\|U-\calU\|_1$ and $\|\calU^{-1}\|_1$ (see \Cref{lemma:bound3-n} and \Cref{lemma:bound1-n}, respectively).

For any $\bv_s,~\bv_t\in Y^d_{D}\setminus Y^d_{\zeta}$, denote
\begin{align}
r_s = \# I_{in}(\bv_s),~~
\alpha_{st}= \# I_{0}(\bv_s-\bv_t), ~s\neq t,
\label{def:rs}
\end{align}
and
\begin{align}
d_m=\min\limits_{\zeta+1\leq s\leq D} \{r_s\},~~
d_M=\max\limits_{\zeta+1\leq s,t\leq D} \{\alpha_{st} \},
\label{def:dm}
\end{align}
with $0\leq d_m\leq d-1$ and $0< d_M \leq d-1$.
Denote $L_{min}=\min\limits_{1\leq j\leq d} L_j$, $L_{max}=\max\limits_{1\leq j\leq d}L_j$ and $G_{min}=\min\limits_{1\leq j\leq d}G_j$.
In the following analysis, we assume that $\bL$ and $\bG$ satisfy \Cref{assum:L-N}.
\begin{assumption}
\label{assum:L-N}
For given positive numbers $\epsilon,~\epsilon_r~(0\leq r\leq d-1 )$,
assume that $\bL$ and $\bG$ satisfy
\begin{align*}
L_{min}>&\frac{(2N)^{2+\tau}}{C_a}\cdot
\left\{ \Big(\eta+\frac{1}{2} \Big) +
\max\Big\{1, \Big[\frac{(\eta !)^2}{\pi}\Big]^{\frac{1}{2\eta+1}},
\Big [\frac{\pi^{d_M}}{\epsilon}\Big]^
{\frac{1}{(2\eta+1)(d-d_M)}}
\Big\}  \right\}  ,	
\\
G_{min}>&\max_{0\leq r\leq d-1}
\left\{2L_{max}\|\bP\|_1 N+\Big(\eta+\frac{1}{2}\Big)
+\bigg[\frac{1}{\eta}\Big(\frac{C_d^r\cdot
\sum^{d-r}_{\beta=0}(2\eta)^{\beta}C^{\beta}_{d-r}}
{\epsilon_r}\Big)^{\frac{1}{d-r}}
\bigg]^{\frac{1}{2\eta+1}}
\right\},
\end{align*}
where $\eta\geq 1$, $C_d^r=\dfrac{d!}{r!(d-r)!}$
and $d_M$ is defined by \eqref{def:dm}.
\end{assumption}

\begin{lemma}
\label{M-nonsingular}
Under \Cref{ass:DiopCond},
for given positive numbers $\epsilon$ and $\epsilon_r~(0\leq r\leq d-1 )$ such that
\begin{align}
(D-\zeta)\sum_{r=0}^{d-1}
\frac{(\eta !)^{-2r}}{\pi^{-r}}
\epsilon_r
+(D-\zeta-1)\epsilon
<\frac{3^d\pi^{d}}{5^d} \left(\frac{1}{2}+\eta \right)^{-2\eta d},
\label{lem:con-all}
\end{align}
and $\bL,\bG$ satisfy \Cref{assum:L-N}.
Then matrices $U$ and $\calU$ defined in \eqref{dey:ucalu}
are nonsingular. Moreover, the inequality $\|\calU^{-1}\|_1\cdot \|U-\calU\|_1<1$ holds.
\end{lemma}

The proof of \Cref{M-nonsingular} will be presented in the end at this subsection.


For convenience, given a vector $\bell=(\ell_j)_{j=1}^d\in\bbZ_*^d$, denote
\begin{align*}
&A^0(v_{s,j},h_{t,j})
= \frac{(\eta !)^{2}\vert e^{i2\pi(v_{s,j}-h_{t,j})}-1\vert}
{2\pi\Pi^{\eta }_{j_1=-\eta }\vert(v_{s,j}-h_{t,j})+j_1\vert},
~~~~~~~~~~~~~~~j \in I_{0}(\bell) \cap I^c_{0}(\bv_s-\bh_t),\\
&A^1(v_{s,j},h_{t,j},\ell_j)
=\frac{(\eta !)^{2}\vert e^{i2\pi(v_{s,j}-h_{t,j}-\ell_j G_j)}-1\vert}
{2\pi\Pi^{\eta }_{j_1=-\eta }\vert(v_{s,j}-h_{t,j}-\ell_j G_j)+j_1\vert},
~~j \in I^c_{0}(\bell) \cap I^c_{0} (\bv_s-\bh_t-\bell\circ \bG).
\end{align*}
Obviously,
$A^1(v_{s,j},h_{t,j},\ell_j)=A^0(v_{s,j},h_{t,j})$ for $j\in I_{0}(\bell)$.
Next, we prove that $A^0(v_{s,j},h_{t,j})$ and $A^1(v_{s,j},h_{t,j},\ell_j)$
are well-defined and bounded under \Cref{ass:DiopCond} and \Cref{assum:L-N}.
	
\begin{proposition}
\label{pro:boundA0A1}
Under \Cref{ass:DiopCond} and \Cref{assum:L-N}. For a given vector
$\bell=(\ell_j)_{j=1}^d\in\bbZ_*^d$,
the following conclusions hold

(1) $A^0(v_{s,j},h_{t,j})$ and
$A^1(v_{s,j},h_{t,j},\ell_j)$ are well-defined. Moreover, for $j\notin I_{0}(\bell)$,
we have $j\notin I_{0}(\bv_s-\bh_t-\bell \circ \bG)$;

(2) If $I_{in}(\bv_s)=\emptyset$.
Then, $I_{0}(\bv_s-\bh_s)=\emptyset$,
\begin{align}
	A^0(v_{s,j},h_{t,j})<1, ~~1\leq j\leq d,
	\label{bound:Ira_empty_A0}
\end{align}
and
\begin{align}
A^1(v_{s,j},h_{t,j},\ell_j)
\leq \frac{(\eta!)^2}
{\pi(\vert \ell_j\vert G_j-\|\bv_{s}-\bh_{t}\|_{\ellinf}-\eta)^{2\eta+1}},
~~j\notin I_{0}(\bell).
\label{bound:Ira_empty_A12}
\end{align}
In particular, when $s\neq t$ and $j\notin I_{0}(\bv_s-\bv_t)$, it follows that
\begin{align}
A^0(v_{s,j},h_{t,j})<
\frac{(\eta !)^{2}}
{\pi(\frac{L_{min}C_a}{(2N)^{2+\tau}}-\frac{1}{2}-\eta)^{(2\eta+1)}}\leq 1.
\label{bound:A0_special}
\end{align}
		
(3) If $I_{in}(\bv_s)\neq \emptyset$, that is $\# I_{in}(\bv_s)=r_s>0$.
Then, $I_{0}(\bv_s-\bh_s)\neq \emptyset$.

(3.1) When $s=t$, we have
\begin{align*}
	A^0(v_{s,j},h_{s,j})<1,~j\notin I_{0}(\bv_s-\bh_s).
\end{align*}
		
When $s\neq t$, if
$I_{0}(\bv_s-\bh_s)\cap I^c_0(\bv_s-\bv_t)\neq \emptyset$,
we have
\begin{align}
A^0(v_{s,j},h_{t,j})
=0,~~ j\in I_{0}(\bv_s-\bh_s)\cap I^c_0(\bv_s-\bv_t).
\label{eq:sneqtA0}
\end{align}
Otherwise,
$I_{0}(\bv_s-\bh_s)\cap I^c_0(\bv_s-\bv_t)= \emptyset$,
then $I_{0}(\bv_s-\bh_s)\subset I_{0}(\bv_s-\bv_t)$ and
\begin{align}
\begin{cases}
A^0(v_{s,j},h_{t,j})
<1,~~ j\notin I_{0}(\bv_s-\bv_t),\\
A^0(v_{s,j},h_{t,j})=A^0(v_{s,j},h_{s,j})
<1,~~ j\in I_{0}(\bv_s-\bv_t)\setminus I_{0}(\bv_s-\bh_s).
\label{bound:Ira_nonempty_A0}
\end{cases}
\end{align}

(3.2) If $I_{0}(\bv_s-\bh_s)\cap I^c_0(\bell)\neq \emptyset$,
we have
\begin{align}
A^1(v_{s,j},h_{t,j},\ell_j)=0,
~j\in I_{0}(\bv_s-\bh_s)\cap I^c_0(\bell).
\label{bound:Ira_empty_A1}
\end{align}
Otherwise, $I_{0}(\bv_s-\bh_s)\cap I^c_0(\bell)= \emptyset$,
we have $I_{0}(\bv_s-\bh_s)\subset I_{0}(\bell)$ and
\begin{align}
A^1(v_{s,j},h_{t,j},\ell_j)
\leq \frac{(\eta!)^2}
{\pi(\vert \ell_j\vert G_j -\|\bv_{s}-\bh_{t}\|_{\ellinf}-\eta)^{2\eta+1}},
~~j\notin I_{0}(\bell).
\label{bound:Ira_nonempty_A1}
\end{align}
\end{proposition}
	
	\begin{proof}
		(1) From \Cref{assum:L-N}, the inequality \eqref{con:pro-NL} holds.
		Then, according to the proof of \Cref{pro:M11M21},
		the conclusion is easy to prove.
		
		
(2) $I_{in}(\bv_s)=\emptyset~(\zeta+1\leq s\leq D)$
implies $I_{0}(\bv_s-\bh_t)=\emptyset$.
When $s=t$,
\begin{align*}
A^0(v_{s,j}, h_{s,j})=
\frac{(\eta !)^{2}
	\vert e^{i2\pi(v_{s,j}-h_{s,j})}-1\vert}
{2\pi\Pi^{\eta }_{j_1=-\eta }
	\vert (v_{s,j}-h_{s,j})+j_1\vert}
&=\frac{(\eta !)^{2}\cdot
	2 \sin[\pi (v_{s,j}-h_{s,j})]}
{2\pi \vert v_{s,j}-h_{s,j} \vert \cdot
	\Pi^{\eta }_{j_1=1 }
	\vert (v_{s,j}-h_{s,j})^2-j_1^2\vert}\\
&<\frac{\sin [\pi (v_{s,j}-h_{s,j})]}
{\pi \vert v_{s,j}-h_{s,j} \vert}\cdot
\frac{1}
{\Pi^{\eta }_{j_1=1}
	[1-\frac{(v_{s,j}-h_{s,j})^2}{j_1^2}]}.
\end{align*}
From the Weierstrass factorization formula
\begin{align*}
	\sin(\pi z)=\pi z\mathop{\prod}_{\ell=1}^{\infty}
	\Big(1-\frac{z^2}{\ell^2}\Big),
	~~z\in \bbC,
\end{align*}
the function
$$F_1(x)=\mathop{\prod}^{\eta}_{j_1=1}
\Big(1-\frac{x^2}{j_1^2}\Big)$$
monotonically decreases with respect to $\eta$ when $0<x < 1/2$.
Therefore,
\begin{align*}
A^0(v_{s,j}, h_{s,j})<1.
\end{align*}
When $s\neq t$, we also show that inequality \eqref{bound:A0_special} holds.
For $j\notin I_{0}(\bv_s-\bv_t)$, applying the inequality \eqref{bound:A0vshs},
we can obtain
\begin{align*}
A^0(v_{s,j}, h_{t,j})=\frac{(\eta !)^{2}
	\vert e^{i2\pi(v_{s,j}-h_{t,j})}-1\vert}
{2\pi\Pi^{\eta }_{j_1=-\eta }
	\vert v_{s,j}-h_{t,j}+j_1\vert}
&\leq \frac{(\eta !)^{2}}
{\pi\Pi^{\eta }_{j_1=-\eta }
	\vert v_{s,j}-h_{t,j}+j_1 \vert}\\
&<\frac{(\eta !)^{2}}{\pi}\cdot
\frac{1}{\big(\frac{L_{min}C_a}{(2N)^{2+\tau}}-\frac{1}{2}-\eta\big)^{2\eta+1}}
\leq 1.
\end{align*}
For $j\in I_{0}(\bv_s-\bv_t)$,
it follows that
\begin{align*}
	A^0(v_{s,j}, h_{t,j})=A^0(v_{s,j}, h_{s,j})<1.
\end{align*}
		
Meanwhile, from the definition of $A^1(v_{s,j},h_{t,j},\ell_j)$
and
\begin{align*}
|v_{s,j}-h_{t,j}-\ell_j G_j+j_1|
\geq |\ell_j|G_{min}-\|\bv_{s}-\bh_{t}\|_{\ellinf}-\eta,
\end{align*}
then inequality \eqref{bound:Ira_empty_A12} holds.
		
(3) $I_{in}(\bv_s)\neq \emptyset~~ (\zeta+1\leq s\leq D)$ implies $I_{0}(\bv_s-\bh_s)\neq \emptyset$.
		
(3.1) For $s=t$, similar to the above analysis in (2), we have
\begin{align*}
A^0(v_{s,j},h_{s,j})<1,~j\notin I_{0}(\bv_s-\bh_s).
\end{align*}
		
For $s\neq t$, we consider two cases. If
$I_{0}(\bv_s-\bh_s)\cap I^c_0(\bv_s-\bv_t)\neq \emptyset$,
there exists $j$ such that $v_{s,j}$ is an integer and $v_{s,j}-h_{t,j}\in \bbZ_*$. This means $A^0(v_{s,j},h_{t,j})=0$. Otherwise, $I_{0}(\bv_s-\bh_s)\cap I^c_0(\bv_s-\bv_t)=\emptyset$ implies $I_{0}(\bv_s-\bh_s)\subset I_{0}(\bv_s-\bv_t)$.
When $j\notin I_{0}(\bv_s-\bv_t)$, we have $v_{s,j}-h_{t,j}\notin \bbZ$ and
$A^0(v_{s,j},h_{t,j})< 1$.
		
(3.2) Similar to the proof of the conclusion (2), \eqref{bound:Ira_empty_A1} and \eqref{bound:Ira_nonempty_A1} can be proved.
\end{proof}

Next, we estimate the bound of $\|U-\calU\|_1$ by the relation between DFT and NWFT. Denote
\begin{align*}
g_0(t_0)=\frac{1}
{(\frac{L_{min}C_a}{(2N)^{2+\tau}}-\frac{1}{2}-\eta)^{t_0}},~~
g_1(t_1,t_2)=\frac{1}
{(t_1 G_{min}-\|\bv_{s}-\bh_{t}\|_{\ellinf}-\eta)^{t_2}}.
\end{align*}
Inequalities \eqref{bound:A0_special} and \eqref{bound:Ira_nonempty_A1} are rewritten as
\begin{align*}
A^0(v_{s,j},h_{t,j})<\frac{(\eta !)^{2}}{\pi}g_0(2\eta+1)\leq 1,~~
A^1(v_{s,j},h_{t,j},\ell_j)\leq \frac{(\eta!)^2}{\pi}g_1(\vert \ell_j\vert,2\eta+1).
\end{align*}
Denote
\begin{align*}
g_2 = \bigg(\frac{1}{\eta (G_{min}-\frac{1}{2}-\eta)^{2\eta+1}}\bigg)^{d-r},~~
g_3 = \bigg(\frac{1}{\eta
(G_{min}-2L_{max}\|\bP\|_1N-\frac{1}{2}-\eta)^{2\eta+1}}\bigg)^{d-r}.
\end{align*}
	
\begin{theorem}
\label{lemma:bound3-n}
Under \Cref{ass:DiopCond}, then
\begin{align}
\nonumber
\|U-\calU\|_1
&< \sum^{d-1}_{r=d_m}
\frac{(\eta !)^{2(d-r)}}{\pi^{d-r}}
C_{d-d_m}^{r-d_m}\cdot
\sum^{d-r}_{\beta=0}(2\eta)^{\beta}C^{\beta}_{d-r}[g_2+(D-\zeta-1)g_3],
\end{align}
where $d_m$ is defined in \eqref{def:dm}.
\end{theorem}

\begin{proof}
Let's prove this theorem for $I_{in}(\bv)= \emptyset$ and $I_{in}(\bv) \neq \emptyset$, respectively.

(1) When $I_{in}(\bv)= \emptyset$, from \Cref{assum:L-N}, we have
\begin{align*}
\frac{L_{min}C_a}{(2N)^{2+\tau}}-\frac{1}{2}-\eta
\geq \Big(\frac{(\eta !)^2}{\pi}\Big)^{\frac{1}{2\eta+1}},~~
G_{min}-2L_{max}\|\bP\|_1N - 1/2>\eta.
\end{align*}
For $s,t=\zeta+1,\zeta+2,\cdots,D$,
from the definition of $J_r$ in \eqref{def:J_r} and \Cref{lemma:relation-DFT-NWFT-n}, we have
\begin{align}
\vert u_{st}- \fraku_{st}\vert
&=\vert F_{e^{i2\pi\bv_t\cdot\bx/\bL},\bL,\bG}^{\eta }(\bh_s)
-\phi_{e^{i2\pi\bv_t\cdot\bx/\bL},\bL}^{\eta}
(\bh_s/\bL)\vert\\
&~~\leq \sum_{\bell\in\bbZ_*^d}
\Big\vert\phi^{\eta }_{e^{i2\pi\bv_t\cdot\bx/\bL},\bL}
\Big(\frac{\bh_s+\bell\circ \bG}{\bL}\Big)\Big\vert\notag\\
&= \sum_{r=0}^{d-1}
\sum_{\bell\in J_{r}}
\Bigg\{
\bigg(\mathop{\prod}_{j\in I_{0}(\bell)}
\frac{(\eta !)^{2}
\vert e^{i2\pi(v_{t,j}-h_{s,j})}-1\vert}
{2\pi \Pi^{\eta }_{j_1=-\eta }|v_{t,j}-h_{s,j}+j_1|}\bigg)\\
&~~\cdot
\bigg(\mathop{\prod}^{d}_{j=1,j\notin I_{0}(\bell)}
\frac{(\eta !)^{2}\vert e^{i2\pi(v_{t,j}-h_{s,j}-\ell_j G_j)}-1\vert}
{2\pi\Pi^{\eta }_{j_1=-\eta }
\vert v_{t,j}-h_{s,j}-\ell_j G_j+j_1\vert}\bigg)\Bigg\}\notag \\
&=\sum_{r=0}^{d-1}
\sum_{\bell\in J_{r}}
\Bigg\{
\mathop{\prod}_{j\in I_{0}(\bell)}
A^0(v_{t,j}, h_{s,j})
\cdot
\mathop{\prod}^{d}_{j=1,j\notin I_{0}(\bell)}
A^1(v_{t,j}, h_{s,j},\ell_j)\Bigg\}.
\label{eq:bound-m-1}
\end{align}
In the following, we will analyze the bound of \eqref{eq:bound-m-1} for $s\neq t$ and $s=t$, respectively.

When $s\neq t$, according to
\begin{align*}
\|\bv_{s}-\bh_{t}\|_{\ellinf}
\leq \|\bv_{s}-\bv_{t}\|_{\ellinf}+
\|\bv_{t}-\bh_{t}\|_{\ellinf}
&< 2L_{max}\|\bP\|_1N + 1/2,
\end{align*}
and from the conclusion (2) of \Cref{pro:boundA0A1}, we have
\begin{align*}
&\sum_{\bell\in J_{r}}
\Bigg\{
\mathop{\prod}_{j\in I_{0}(\bell)}
A^0(v_{t,j}, h_{s,j})
\cdot
\mathop{\prod}^{d}_{j=1,j\notin I_{0}(\bell)}
A^1(v_{t,j}, h_{s,j},\ell_j)\Bigg\}\\
&< \sum_{\bell\in J_r}
\mathop{\prod}^{d}_{j=1,j\notin I_{0}(\bell)}
\frac{(\eta!)^2}{\pi}g_1(\vert \ell_j\vert,2\eta+1)\\
&=\Big(\frac{2(\eta !)^{2}}{\pi} \Big)^{d-r} C^r_d\cdot
\bigg\{
\sum_{\substack{\bell\in\bbZ_*^d,
\bell>0,\\
\ell_{1}=\cdots=\ell_{r}=0}}
\mathop{\prod}^{d}_{j=r+1}g_1( \ell_j,2\eta+1)
\bigg\}\\
&=\Big(\frac{2(\eta !)^{2}}{\pi} \Big)^{d-r}C^r_d\cdot
\bigg\{\sum^{d-r}_{\beta=0}\Big[ C^{\beta}_{d-r}g_1(1,(2\eta+1)\beta)
\sum_{\substack{\bell\in\bbZ_*^d\\ \ell_{r+1}=\cdots=\ell_{r+\beta}=1
\\ \ell_j\geq 2,j=r+\beta+1,\cdots,d}}
\mathop{\prod}^{d}_{j=r+\beta+1} g_1(\ell_j,2\eta+1)\Big]\bigg\}\\
&\leq \Big(\frac{2(\eta !)^{2}}{\pi} \Big)^{d-r}C^r_d\cdot
\Bigg\{\sum^{d-r}_{\beta=0}\Big[ C^{\beta}_{d-r}
g_1(1,(2\eta+1)\beta)
\mathop{\prod}^{d}_{j=r+\beta+1}
\int^{\infty}_1 g_1(y_j,2\eta+1)\,dy_j\Big] \Bigg\}\\
&=\Big(\frac{2(\eta !)^{2}}{\pi} \Big)^{d-r}C^r_d\cdot
\Bigg\{\sum^{d-r}_{\beta=0}\Big[ C^{\beta}_{d-r}
g_1(1,(2\eta+1)\beta)\cdot
\Big(\frac{1}{2\eta
G_{min}(G_{min}-\|\bv_{t}-\bh_{s}\|_{\ellinf}-\eta)^{2\eta}}\Big)^{d-r-\beta}
\Big]\Bigg\}\\
&\leq \Big(\frac{2(\eta !)^{2}}{\pi} \Big)^{d-r}C^r_d\cdot
\Bigg\{\sum^{d-r}_{\beta=0}\Big[ C^{\beta}_{d-r}
g_1(1,(2\eta+1)\beta)\cdot
\Big(\frac{1}{2\eta
(G_{min}-\|\bv_{t}-\bh_{s}\|_{\ellinf}-\eta)^{2\eta+1}}\Big)^{d-r-\beta}\Big]\Bigg\}\\
&< \frac{(\eta !)^{2(d-r)}}{\pi^{d-r}}C^r_d\cdot
\sum^{d-r}_{\beta=0}(2\eta)^{\beta}C^{\beta}_{d-r}
\Bigg[\frac{1}{\eta
(G_{min}-2L_{max}\|\bP\|_1N - \frac{1}{2}-\eta)^{2\eta+1}}\Bigg]^{d-r}.
\end{align*}

Therefore,
\begin{align*}
	\vert u_{st}- \fraku_{st}\vert
	< \sum^{d-1}_{r=0}
	\frac{(\eta !)^{2(d-r)}}{\pi^{d-r}}
	C_d^r\cdot
	\sum^{d-r}_{\beta=0}(2\eta)^{\beta}C^{\beta}_{d-r}g_3.
\end{align*}

Similarly, when $s=t$,
\begin{align}
	\vert u_{ss}- \fraku_{ss}\vert
	<\sum^{d-1}_{r=0}
	\frac{(\eta !)^{2(d-r)}}{\pi^{d-r}}
	C_d^r\cdot
	\sum^{d-r}_{\beta=0}(2\eta)^{\beta}C^{\beta}_{d-r}g_2.
	\label{ineq:tt}
\end{align}
It follows that
\begin{align*}
	\|U-\calU\|_1
	&=\max_{\zeta+1\leq t\leq D}\sum^{D}_{s=\zeta+1}
	\vert u_{st}-\fraku_{st}\vert\\
	&< \sum^{d-1}_{r=0}
	\frac{(\eta !)^{2(d-r)}}{\pi^{d-r}}
	C_d^r\cdot\sum^{d-r}_{\beta=0}(2\eta)^{\beta}C^{\beta}_{d-r}
	[g_2+(D-\zeta-1)g_3].
\end{align*}

(2) When $I_{in}(\bv)=I_{0}(\bv_s-\bh_s) \neq \emptyset$, we can obtain a more tighter bound of
$\|U-\calU\|_1$. We consider the diagonal case of $s=t$
and the non-diagonal one of $s\neq t$, separately.

(i) Let's consider the case of $s=t$.
For $r_s \leq r\leq d-1$, we give an index set $\{j^*_1,~ j^*_2,\cdots, j^*_r\}$
such that
$$ I_{0}(\bv_s-\bh_s) \subset \{j^*_1,~ j^*_2,\cdots, j^*_r\}
\subset \{1,2,\cdots,d\},$$
and denote
\begin{align*}
J^*_r=\{\bell\in\bbZ^d: \ell_{j^*_1}=\ell_{j^*_2}=\cdots=\ell_{j^*_r}=0\}.
\end{align*}
From \eqref{eq:NWFT-complex-exponential-n},
we have
\begin{align*}	
&\bigg\vert\frac{1}{L_j}\int^{L_j}_0
\frac{\eta !}{(2\eta -1)!!}
\left(1-\cos\frac{2\pi x_{j}}{L_j}\right)^{\eta}
e^{i2\pi(v_{s,j}-h_{s,j})x_j/L_j}\,dx_j\bigg\vert\\
&=\begin{cases}
1,&j\in I_{0}(\bv_s-\bh_s),\\
A^0(v_{s,j}, h_{s,j}),&j\notin I_{0}(\bv_s-\bh_s).
\end{cases}
\end{align*}
According to the conclusion (3) of \Cref{pro:boundA0A1},
we only need consider the case
$I_{0}(\bv_s-\bh_s)\cap I^c_0(\bell)= \emptyset$
for $\bell\in\bbZ_*^d$
since
\begin{align*}
	A^1(v_{s,j},h_{s,j},\ell_j)=0,
	~j\in I_{0}(\bv_s-\bh_s)\cap I^c_0(\bell).
\end{align*}
Therefore, $I_{0}(\bv_s-\bh_s)\subseteq I_{0}(\bell)$ and
$ \#I_{0}(\bv_s-\bh_s)\leq \#I_{0}(\bell)$.
From inequalities \eqref{bound:Ira_nonempty_A0}
and \eqref{bound:Ira_nonempty_A1}, we can obtain
\begin{align*}
	\vert u_{ss}- \fraku_{ss}\vert
	&\leq \sum_{r=r_s}^{d-1}
	\sum_{\bell\in J_{r}}
	\Bigg\{
	\mathop{\prod}_{j\in I_{0}(\bell)\setminus I_{0}(\bv_s-\bh_s) }
	A^0(v_{s,j},h_{s,j})
	\cdot
	\mathop{\prod}^{d}_{j=1,j\notin I_{0}(\bell)}
	A^1(v_{s,j},h_{s,j},\ell_j)\Bigg\}\\
	&\leq \sum_{r=r_s}^{d-1}
	\sum_{\bell\in J_{r}}
	\Bigg\{
	\mathop{\prod}^{d}_{j=1,j\notin I_{0}(\bell)}
	A^1(v_{s,j},h_{s,j},\ell_j)\Bigg\}\\
	&\leq \sum_{r=r_s}^{d-1} 2^{d-r}
	C^{r-r_s}_{d-r_s}\cdot
	\sum_{\substack{\bell\in J^*_{r}\\\bell>0}}
	\Bigg\{
	\mathop{\prod}^{d}_{\substack{j=1\\ j\notin \{j^*_1,~ j^*_2,\cdots, j^*_r\}}}
	\frac{(\eta !)^{2}}
	{\pi}g_1(\ell_j,2\eta+1)\Bigg\}\\
	&< \sum_{r=r_s}^{d-1}
	\frac{(\eta !)^{2(d-r)}}{\pi^{d-r}}
	C^{r-r_s}_{d-r_s}\cdot
	\sum^{d-r}_{\beta=0}(2\eta)^{\beta}C^{\beta}_{d-r}g_2.
\end{align*}
Obviously, the upper bound mentioned above is tighter than the bound given in \eqref{ineq:tt} since
\begin{align*}
	\sum_{r=r_s}^{d-1}
	\frac{(\eta !)^{2(d-r)}}{\pi^{d-r}}
	C^{r-r_s}_{d-r_s}\cdot
	\sum^{d-r}_{\beta=0}(2\eta)^{\beta}C^{\beta}_{d-r}g_2&\leq \sum_{r=d_m}^{d-1}
	\frac{(\eta !)^{2(d-r)}}{\pi^{d-r}}
	C^{r-d_m}_{d-d_m}\cdot
	\sum^{d-r}_{\beta=0}(2\eta)^{\beta}C^{\beta}_{d-r}g_2\\
	&\leq \sum_{r=0}^{d-1}
	\frac{(\eta !)^{2(d-r)}}{\pi^{d-r}}
	C_d^r\cdot
	\sum^{d-r}_{\beta=0}(2\eta)^{\beta}C^{\beta}_{d-r}g_2.
\end{align*}

(ii) When $s\neq t$ and from \eqref{bound:Ira_nonempty_A1},
the upper bound of $\vert u_{ts}- \fraku_{ts}\vert$ is
\begin{align*}
	\vert u_{ts}- \fraku_{ts}\vert
	< \sum_{r=r_s}^{d-1}
	\frac{(\eta !)^{2(d-r)}}{\pi^{d-r}}
	C^{r-r_s}_{d-r_s}\cdot
	\sum^{d-r}_{\beta=0}(2\eta)^{\beta}C^{\beta}_{d-r} g_3.
\end{align*}
The proof is completed.
\end{proof}

The bound $\Vert U-\calU\Vert_1$ is always finite when $G_j\gg L_j$.
	%
	%
\Cref{M-nonsingular} has presented a sufficient condition to guarantee the
invertibility of $\calU$ and $\|\calU^{-1}_D\|_1\cdot \|\calU_O\|_1 <1$. In the following, we derive an upper bound for $\|\calU^{-1}\|_1$ by decomposing $\calU$ into diagonal part $\calU_D$ and non-diagonal part $\calU_O$.
	
\begin{theorem}
Under \Cref{ass:DiopCond}, then $\calU$ is invertible and
\begin{align}
\|\calU^{-1}\|_1\leq \frac{x_1}{1-x_1x_2},
\label{eq:bound-truncation-NWFT-n}
\end{align}
where
\begin{align*}
\begin{cases}
x_1=\frac{5^{d-d_m}}{3^{d-d_{m}}(\eta !)^{2 (d-d_m) }}(\frac{1}{2}+\eta)^{2\eta (d-d_m)},\\
x_2=(D-\zeta-1)\frac{(\eta !)^{2d}}{\pi^{(d-d_M)}}g_1((2\eta+1)(d-d_M)),
\end{cases}
\end{align*}
$d_m$ and $d_M$ are defined in \eqref{def:dm}.
\label{lemma:bound1-n}
\end{theorem}

\begin{proof}
\Cref{M-nonsingular} implies that $\calU$ is invertible.
For $\bv_t\in\bbR^d\backslash\bbQ^d$,
\begin{align*}
	\fraku_{tt}=\phi_{e^{i2\pi\bv_t\cdot\bx/\bL},\bL}^{\eta}
	(\bh_t/\bL)\neq 0.
\end{align*}
This means that $\calU_D$ is invertible.
Let's prove the inequality \eqref{eq:bound-truncation-NWFT-n} for $I_{in}(\bv)= \emptyset$ and $I_{in}(\bv) \neq \emptyset$, respectively.	

(1) When $I_{in}(\bv)= \emptyset$, from the definition \eqref{eq:NWFT-complex-exponential-n} of $\phi$, we have
\begin{align*}
\vert \fraku_{tt}^{-1}\vert
=\Big\vert [\phi_{e^{i2\pi\bv_t\cdot\bx/\bL},\bL}^{\eta }
(\bh_t/\bL)]^{-1}\Big\vert
&=\Big\vert \mathop{\prod}^{d}_{j=1}
\frac{i2\pi\Pi^{\eta }_{j_1=-\eta }[(v_{t,j}-h_{t,j})+j_1]}
{(-1)^{\eta }(\eta !)^{2}(e^{i2\pi(v_{t,j}-h_{t,j})}-1)}\Big\vert\\
&=\mathop{\prod}^{d}_{j=1}
\frac{2\pi\Pi^{\eta }_{j_1=-\eta }
	\Big\vert v_{t,j}-h_{t,j}+j_1\Big\vert}
{(\eta !)^{2}\Big\vert 2\sin[\pi(v_{t,j}-h_{t,j})]\Big\vert}\\
&\leq \frac{1}{(\eta !)^{2d}}
(\|\bv_{t}-\bh_{t}\|_{\ellinf}+\eta)^{2\eta d}
\Big(\frac{1+\|\bv_{t}-\bh_{t}\|^{2}_{\ellinf}}
{1-\|\bv_{t}-\bh_{t}\|^{2}_{\ellinf}}\Big)^d.
\end{align*}
The last inequality is true according to the following inequality 
\begin{align*}
\frac{\pi^2-x^2}{\pi^2+x^2}
\leq\frac{\sin x}{x}
\leq \Big(\frac{\pi^2-x^2}{\pi^2+x^2}\Big)^{\pi^2/12},
~x\in (0,\pi).
\end{align*}
Since $\calU_D^{-1}=(\fraku_{tt}^{-1})$
and the function
$f(x)=(1+x)/(1-x)$ is monotonically increasing,
it follows that
\begin{align*}
\|\calU_D^{-1}\|_1
=\max_{\zeta+1\leq t\leq D} |\fraku_{tt}^{-1}|
&\leq \max_{\zeta+1\leq t\leq D}
\frac{1}{(\eta !)^{2d}}
(\|\bv_{t}-\bh_{t}\|_{\ellinf}+\eta)^{2\eta d}
\Big(\frac{1+\|\bv_{t}-\bh_{t}\|^{2}_{\ellinf}}
{1-\|\bv_{t}-\bh_{t}\|^{2}_{\ellinf}}\Big)^d\\
&<\frac{5^d}{3^d(\eta !)^{2d}}
(\frac{1}{2}+\eta)^{2\eta d}.
\end{align*}

For the non-diagonal elements of $\calU$,
from \eqref{bound:Ira_empty_A0} and \eqref{bound:A0_special},
we can obtain
\begin{align*}
\vert \fraku_{st}\vert=\Big\vert \phi_{e^{i2\pi\bv_t\cdot\bx/\bL},\bL}^{\eta } (\bh_s/\bL)\Big\vert
&=\Bigg\vert \mathop{\prod}^{d}_{j=1}
\frac{(-1)^{\eta }(\eta !)^{2}[e^{i2\pi(v_{t,j}-h_{s,j})}-1]}
{i2\pi\Pi^{\eta }_{j_1=-\eta }[(v_{t,j}-h_{s,j})+j_1]}
\Bigg\vert\notag\\
&=\mathop{\prod}^{d}_{\substack{j=1,\\j\notin I_{0}(\bv_s-\bv_t)}}
A^0(v_{t,j},h_{s,j})
\cdot
\mathop{\prod}_{j\in I_{0}(\bv_s-\bv_t) }
A^0(v_{s,j},h_{s,j})\notag\\
&<\frac{(\eta !)^{2(d-\alpha_{st})}}{\pi^{d-\alpha_{st}}}
g_0((2\eta+1)(d-\alpha_{st})),
\end{align*}
where $\alpha_{st}$ is defined in \eqref{def:rs}.

From \Cref{assum:L-N}, we can obtain
\begin{align*}
\frac{L_{min}C_a}{(2N)^{2+\tau}}-\frac{1}{2}-\eta\geq 1.
\end{align*}
Therefore, we have
\begin{align*}
\|\calU_O\|_1
=\max_{\zeta+1\leq t\leq D}\sum^{D}_{s=\zeta+1,s\neq t}
|\fraku_{st}|
&< \max_{\zeta+1\leq t\leq D}\sum^{D}_{s=\zeta+1,s\neq t}
\frac{(\eta !)^{2(d-\alpha_{st})}}{\pi^{d-\alpha_{st}}}
g_0((2\eta+1)(d-\alpha_{st}))\\
&<( D-\zeta-1) \frac{(\eta !)^{2d}}{\pi^{d-d_M}}
g_0((2\eta+1)(d-d_M)),
\end{align*}
with $d_M=\max\limits_{s,t} \{\alpha_{st} \}\leq d-1$.
Applying \Cref{lemma:invertible-matrix-bound-n},
the conclusion of this theorem holds for $I_{in}(\bv)= \emptyset$.

(2) When $I_{in}(\bv)\neq \emptyset$,
we can also provide the upper bound on $\|\calU^{-1}\|_1$.

(i) For $s=t$, it follows that
\begin{align*}
\vert \fraku_{ss}^{-1}\vert
=\Big\vert [\phi_{e^{i2\pi\bv_s\cdot\bx/\bL},\bL}^{\eta }
(\bh_s/\bL)]^{-1}\Big\vert
&=\Bigg\vert \mathop{\prod}^{d}_{\substack{j=1\\ j\neq I_{in}(\bv_s)}}
\frac{i2\pi\Pi^{\eta }_{j_1=-\eta }[(v_{s,j}-h_{s,j})+j_1]}
{(-1)^{\eta }(\eta !)^{2}(e^{i2\pi(v_{s,j}-h_{s,j})}-1)}\Bigg\vert\\
&< \frac{5^{d-r_{s}}}{3^{d-r_{s}}(\eta !)^{2(d-r_{s})}}
(\frac{1}{2}+\eta)^{2\eta (d-r_{s})},
\end{align*}
where $r_{s}$ is defined in \eqref{def:rs}.
Due to $\eta\geq 1$, it easily finds
$$1\leq \frac{5}{3(\eta !)^{2}}
(\frac{1}{2}+\eta)^{2\eta},$$
then
\begin{align*}
\frac{5^{d-r_{s}}}{3^{d-r_{s}}(\eta !)^{2(d-r_{s})}}
(\frac{1}{2}+\eta)^{2\eta (d-r_{s})}\leq
\frac{5^{d}}{3^{d}(\eta !)^{2d}}
(\frac{1}{2}+\eta)^{2\eta d}.
\end{align*}
Therefore, we can obtain
\begin{align*}
\|\calU_D^{-1}\|_1
=\max_{\zeta+1\leq s\leq D} |\fraku_{ss}^{-1}|
<\frac{5^{d-d_m}}{3^{d-d_m}(\eta !)^{2{(d-d_m)}}}
(\frac{1}{2}+\eta)^{2\eta (d-d_m)}
\leq\frac{5^d}{3^d(\eta !)^{2d}}
(\frac{1}{2}+\eta)^{2\eta d},
\end{align*}
with $d_m=\min\limits_{ s} \{r_{s}\}\leq d-1$.

(ii) For $s\neq t$ and from \eqref{eq:sneqtA0},
$A^0(v_{t,j},h_{s,j})=0$
for $j\in I_{0}(\bv_t-\bh_t)\cap I^c_0(\bv_s-\bv_t)$,
which implies $ \fraku_{st}=0$.
Therefore, we only consider
$I_{0}(\bv_t-\bh_t)\subset I_{0}(\bv_s-\bv_t)$.
\begin{align*}
\vert \fraku_{st}\vert
&=\mathop{\prod}^{d}_{\substack{j=1,\\j\notin I_{0}(\bv_s-\bv_t)}}
A^0(v_{t,j},h_{s,j})
\cdot
\mathop{\prod}_{j\in I_{0}(\bv_s-\bv_t)\setminus I_{in}(\bv_s)}
A^0(v_{s,j},h_{s,j}) \notag\\
&< \frac{(\eta !)^{2d}}{\pi^{(d-d_M)}}
g_0((2\eta+1)(d-d_M)),
\end{align*}
with $d_M=\max\limits_{ s,t} \{\alpha_{st} \}\leq d-1$.
Consequently, we have
\begin{align*}
\|\calU_O\|_1
< (D-\zeta-1) \frac{(\eta !)^{2d}}{\pi^{d-d_M}}
g_0((2\eta+1)(d-d_M)).
\end{align*}

From the abve analysis, inequality \eqref{eq:bound-truncation-NWFT-n} holds.
\end{proof}

According to \Cref{lemma:bound3-n} and \Cref{lemma:bound1-n},
we can derive the upper bound of $\|U^{-1}\|_1$.
\begin{theorem}
Under \Cref{ass:DiopCond}, $U$ is nonsingular and
\begin{align*}
\|U^{-1}\|_1 \leq\frac{x_1}
{1-x_{1}(x_{2}+x_{3})},
\end{align*}
where
\begin{align*}
&x_1=
\dfrac{5^{d-d_m}}{3^{d-d_m}(\eta !)^{2{(d-d_m)}}}
\Big(\dfrac{1}{2}+\eta \Big)^{2\eta {(d-d_m)}},\\
&x_2=(D-\zeta-1) \frac{(\eta !)^{2d}}{\pi^{d-d_M}}g_0((2\eta+1)(d-d_M)),\\
&x_3=\sum\limits^{d-1}_{r=d_m}
\dfrac{(\eta !)^{2(d-r)}}{\pi^{(d-r)}}
C_{d-d_m}^{r-d_m}\cdot
\sum\limits^{d-r}_{\beta=0}(2\eta)^{\beta}C^{\beta}_{d-r}
[g_2+(D-\zeta-1)g_3],
\end{align*}
$d_m$ and $d_M$ are defined by \eqref{def:dm}.
\label{thm:bound-M-n}
\end{theorem}
	
\begin{proof}
\Cref{M-nonsingular} implies that $U$ is invertible.
According to \Cref{lemma:invertible-matrix-bound-n}, we have
\begin{align*}
\|U^{-1}\|_1
\leq \frac{\|\calU^{-1}\|_1}{1-\|\calU^{-1}\|_1\cdot
\|U-\calU\|_1}:=\bar F(\|\calU^{-1}\|_1,\|U-\calU\|_1).
\end{align*}
$\bar F$ monotonously increases with respect to
$\|\calU^{-1}\|_1$ and $\|U-\calU\|_1$ since $\partial\bar F/\partial (\|\calU^{-1}\|_1)>0$ and $\partial\bar F/\partial (\|U-\calU\|_1)>0$.
Combining with \Cref{lemma:bound3-n} and \Cref{lemma:bound1-n},
the proof can be completed.
\end{proof}


{\bf The proof of \Cref{M-nonsingular}:}	From the above analysis, we know that $x_1(x_2+x_3)<1$ implies $U$ and $\calU$ are invertible.
We will now demonstrate that this inequality is true when \Cref{assum:L-N} is satisfied for $\bL$ and $\bG$.

(1) When $I_{in}(\bv)= \emptyset$, denote $x_3=(D-\zeta-1)C_O+C_D$ where
\begin{align*}
C_O=\sum^{d-1}_{r=0}
\frac{(\eta !)^{2(d-r)}}{\pi^{(d-r)}}
C_d^r\cdot \sum^{d-r}_{\beta=0}(2\eta)^{\beta}C^{\beta}_{d-r}g_3,~~
C_D=\sum^{d-1}_{r=0}
\frac{(\eta !)^{2(d-r)}}{\pi^{(d-r)}}
C_d^r\cdot
\sum^{d-r}_{\beta=0}(2\eta)^{\beta}C^{\beta}_{d-r}g_2,
\end{align*}
and
\begin{align*}
\frac{1}{x_1}=
\frac{3^d(\eta !)^{2d}}{5^d} (\frac{1}{2}+\eta)^{-2\eta d}=C(d,\eta).
\end{align*}
Since $d\geq 1$ and $L_j>0$, then $C_D<C_O$. The inequality $x_1(x_2+x_3)<1$ becomes		
\begin{align}
(D-\zeta-1)C_O+C_D+(D-\zeta-1)\frac{(\eta !)^{2d}}{\pi^{d-d_M}}
g_0((2\eta+1)(d-d_M))
<C(d,\eta).
\label{ineq:L-N}
\end{align}
Hence, we need to prove that
the inequality \eqref{ineq:L-N} holds.
	
Firstly, for a given positive number $\epsilon$,
since $L_{max}$ satisfies
\begin{align*}
\frac{(2N)^{2+\tau}}{C_a}\cdot
\Big\{\Big [\frac{\pi^{d_M}}{\epsilon}\Big]^
{\frac{1}{(2\eta+1)(d-d_M)}}
+\frac{1}{2}+\eta\Big\}
<L_{max},
\end{align*}
we have
\begin{align*}
\Big[\frac{\pi^{d_M}}{\epsilon}\Big]^
{\frac{1}{(2\eta+1)(d-d_M)}}
+\frac{1}{2}+\eta
<\frac{L_{max}C_a}{(2N)^{2+\tau}}.
\end{align*}
It follows that
\begin{align*}
\frac{\pi^{d_M}}
{[\frac{L_{max}C_a}{(2N)^{2+\tau}}-\frac{1}{2}-\eta]^{(2\eta+1)(d-d_M)}}
=\pi^{d_M}g_0((2\eta+1)(d-d_M))
<\epsilon.
\end{align*}
Secondly, for positive numbers
$\epsilon_r$ ($r=0,1,\cdots,d-1$),
since $G_{min}$ satisfies
\begin{align*}
G_{min}>\max_{0\leq r\leq d-1}
\Big\{2L_{max}\|\bP\|_1N+\frac{1}{2}+\eta
+\Big\{\Big[\frac{C_d^r\cdot
\sum^{d-r}_{\beta=0}(2\eta)^{\beta}C^{\beta}_{d-r}}
{\epsilon_r}\Big]^{\frac{1}{d-r}}\cdot
\frac{1}{\eta}
\Big\}^{\frac{1}{2\eta+1}}
\Big\},	
\end{align*}
\textit{i.e.},
\begin{align*}
\Big[\frac{C_d^r\cdot
\sum^{d-r}_{\beta=0}(2\eta)^{\beta}C^{\beta}_{d-r}}
{\epsilon_r}\Big]^{\frac{1}{d-r}}
<\eta (G_{min}-2L_{max}\|\bP\|_1N-\frac{1}{2}-\eta)^{2\eta+1},
\end{align*}
we can obtain
\begin{align*}
C_O&=\sum^{d-1}_{r=0}
\frac{(\eta !)^{2(d-r)}}{\pi^{(d-r)}}
C_d^r\cdot
\sum^{d-r}_{\beta=0}(2\eta)^{\beta}C^{\beta}_{d-r}
\Big[\frac{1}
{\eta (G_{min}-2L_{max}\|\bP\|_1N- \frac{1}{2}-\eta)^{2\eta+1}}\Big]^{d-r}\\
&<\sum_{r=0}^{d-1}
\frac{(\eta !)^{2(d-r)}}{\pi^{(d-r)}}\epsilon_r.
\end{align*}
Moreover, $C_D<C_O$ means that
$C_D<\sum_{r=0}^{d-1}\frac{(\eta !)^{2(d-r)}}{\pi^{(d-r)}}\epsilon_r $.
If positive numbers $\epsilon_r~(r=0,1,\cdots,d-1)$ and
$\epsilon$ satisfy
the inequality \eqref{lem:con-all},
we have
\begin{align*}
&(D-\zeta-1)C_O+C_D+(D-\zeta-1)\frac{(\eta !)^{2d}}{\pi^{d-d_M}}
g_0((2\eta+1)(d-d_M))\\
&<(D-\zeta)\sum_{r=0}^{d-1}
\frac{(\eta !)^{-2r}}{\pi^{-r}}
\epsilon_r+(D-\zeta-1)\epsilon
<C(d,\eta).
\end{align*}
Therefore, the inequality \eqref{ineq:L-N} is true and
$U$ is nonsingular. Moreover, we can obtain
\begin{align*}
(D-\zeta-1)\frac{(\eta !)^{2d}}{\pi^{d-d_M}}
g_0((2\eta+1)(d-d_M))
<C(d,\eta),
\end{align*}
which means $x_1x_2<1$ and $\calU$ is nonsingular.
The condition that guarantees $U$ is nonsingular is sufficient to ensure that $\calU$ is also nonsingular.

(2) When $I_{in}(\bv)\neq \emptyset$, based on the analysis in \Cref{lemma:bound3-n} and \Cref{lemma:bound1-n},
we know that $x_3=(D-\zeta-1)\hat{C}_O+\hat{C}_D$ where
\begin{align*}
&\hat{C}_O=\sum_{r=d_m}^{d-1}
\frac{(\eta !)^{2(d-r)}}{\pi^{d-r}}
C^{r-d_m}_{d-d_m}\cdot
\sum^{d-r}_{\beta=0}(2\eta)^{\beta}C^{\beta}_{d-r}g_3,\\
&\hat{C}_D=\sum_{r=d_m}^{d-1}
\frac{(\eta !)^{2(d-r)}}{\pi^{d-r}}
C^{r-d_m}_{d-d_m}\cdot
\sum^{d-r}_{\beta=0}(2\eta)^{\beta}C^{\beta}_{d-r}g_2.
\end{align*}
Note that $\hat{C}_O \leq C_O$, $\hat{C}_D\leq C_D$
and $\hat{C}(d,\eta)>C(d,\eta)$.
Hence, if the inequality \eqref{ineq:L-N} holds, we have
\begin{align*}
(D-\zeta-1)\hat{C}_O+\hat{C}_D+(D-\zeta-1)\frac{(\eta !)^{2d}}{\pi^{d-d_M}}
g_0((2\eta+1)(d-d_M))
<\hat{C}(d,\eta).
\end{align*}
This means that $U$ and $\calU$ are nonsingular.
Consequently, the proof of \Cref{M-nonsingular} is completed.
	
\textbf{Subproblem 2: The bound of $\|M_{12}\|_1$}
	
\label{sub:error-M12-n}
Based on the analysis in \Cref{sub:error-M-n},
we can directly give the upper bound of $\|M_{12}\|_1$.
Based on the proof of \Cref{lemma:bound3-n} and \Cref{lemma:bound1-n}, it follows that
\begin{align*}
\|M_{12}\|_1<(\zeta+1)(x_2+y_2),
\end{align*}
where
\begin{align}
x_2=\frac{(\eta !)^{2d}}{\pi^{d-d_M}}
g_0((2\eta+1)(d-d_M)),
~y_2=\sum\limits^{d-1}_{r=d_m}
\dfrac{(\eta !)^{2(d-r)}}{\pi^{d-r}}
C_{d-d_m}^{r-d_m}\cdot
\sum\limits^{d-r}_{\beta=0}(2\eta)^{\beta}C^{\beta}_{d-r}g_3,
\label{eq:x2y2}
\end{align}
with $d_m,~d_M$ defined by \eqref{def:dm}.

\subsection{Analysis of Diophantine approximation matrix $ \Delta \bf V$}
\label{subsec:Dio-matrix}

In this subsection, we analyze the approximation rate of Diophantine approximation error $\Vert \DbV\Vert_e$ and discuss the periodic approximation function sequence. 
From the definition of $\Vert \DbV\Vert_e$, we can derive  
\begin{align*}
\Vert \DbV\Vert_e
&=\Vert (\bh_1-\bv_1,\bh_2-\bv_2,\cdots,\bh_{D}-\bv_{D})\Vert_e\\
&= \sum_{\ell=1}^{D} \sum_{j=1}^{d} \vert h_{\ell,j}-v_{\ell,j}\vert
\leq d  \max_{1\leq j \leq d} \sum_{\ell=1}^{D} \vert h_{\ell,j}-v_{\ell,j}\vert,
\end{align*}
which is equivalent to the simultaneous approximation of $\bu_{j}=(v_{\ell,j})^{D}_{\ell=1},~j=1,2,\cdots,d$.
Denote 
\begin{align*}
\hat{R}(\bh)=(h_{\ell,j}-v_{\ell,j})^{D,d}_{\ell=1,j=1},~
R(\bv)=([v_{\ell,j}]-v_{\ell,j})^{D,d}_{\ell=1,j=1}.
\end{align*}
Now we can show that $\Vert \hat{R}(\bh)\Vert_{\ellinf}<1/2$ if and only if $h_{\ell,j}=[v_{\ell,j}]$. When $\hat{R}(\bh)\neq R(\bv)$ and $\Vert R(\bv)\Vert_{\ellinf} <1/2$, we have
\begin{align*}
\Vert\hat{R}(\bh)-R(\bv) \Vert_{\ellinf}
=\Vert(h_{\ell,j}-[v_{\ell,j}])^{D,d}_{\ell=1,j=1} \Vert_{\ellinf}\geq 1.
\end{align*}
This means that $\Vert\hat{R}(\bh) \Vert_{\ellinf}
\geq \Vert\hat{R}(\bh)-R(\bv) \Vert_{\ellinf}-\Vert R(\bv) \Vert_{\ellinf} >1/2$, which is obviously contradictory.

Since we assume that there are only integers and irrational elements in $\bu_{j}~(j=1,2,\cdots, d)$, then $h_{\ell,j}=v_{\ell,j}$ when $v_{\ell,j}$ is an integer. According to \Cref{thm:Dirichletapp} and \Cref{def:bestapprox}, it follows that
\begin{align*}
\Vert \DbV\Vert_e
= \sum_{j=1}^d \sum_{k=1}^{s_j} \vert h_{\ell_k,j}-v_{\ell_k,j}\vert
\leq   \sum_{j=1}^d C_{s_j} L_j^{-1/s_j},~~L_j\in \calT_j (Y^d_{D}),
\end{align*}
where $s_{j}$ represents the number of different irrational Fourier frequency elements in $\bu_{j}~(j=1,2,\cdots, d)$, respectively.
As a consequence, the above expression provides a upper bound of $\|\DbV\|_e$ rather a supremum, and demonstrates that $\|\DbV\|_e$ is inversely proportional to $L_j$. Nevertheless, the uniform decrease of $\|\DbV\|_e$ with a gradually increasing $L_j$ cannot be guaranteed due to the property of irrational number. It is possible that a gradual increase of $L_j$ may increase $\|\DbV\|_e$.

\begin{remark}

(i) When the vector $\bu_{j}$ only contains an irrational number $\lambda_{j_1}$, then  $\{t_1,~t_2,\cdots\}$ is continued fraction expansion of $\lambda_{j_1}$.

(ii)  When the vector $\bu_{j}$ contains more than one irrational number, finding the simultaneous approximation sequence $\{q_{1,j},~q_{2,j},\cdots\}$ is NP-hard problem \cite{Lagarias1985computational,Lagarias1993quality}.
\end{remark}
	

	
\subsection{Summary}
\label{subsec:summary}

We put previous results together and give the bound of the rational approximation error
\begin{equation}
\begin{aligned}
\| f_{p}-\calP_N f\|_{\infty}
&< \underbrace{b_{max}\|M^{-1}\|_1\|M_p-M\|_e
+2\pi b_{max} \|\DbV\|_e }_{\varepsilon_1}\\
&<\underbrace{ 2\pi b_{max} \Big[D \frac{[1+(\zeta+1)(x_2+y_2)] x_1}
{1-x_{1}(x_{2}+x_{3})}
+1\Big]\|\DbV\|_e}_{\varepsilon_2},
\end{aligned}
\label{eq:summary}
\end{equation}
where the definitions of $x_1$, $x_2$, $x_3$ and $y_2$ are given in \Cref{thm:bound-M-n} and \eqref{eq:x2y2}, respectively.



The main result of this work is summarized as
\begin{theorem}
\label{thm:bound of tf-f}
Under \Cref{ass:DiopCond}, and assume that the quasiperiodic function $f\in H_{QP}^\alpha(\bbR^d)$, 
the error in approximating $f$ with $f_{p}\in H^\alpha(\bbT^d)$ is given by
\begin{equation*}
\begin{aligned}
\Vert f_{p}-f\Vert_{\infty}
\lesssim  \max_{1\leq j \leq d} L_j^{-s_j} + N^{\kappa-\alpha}\vert f \vert_\alpha,
\end{aligned}
\end{equation*} 
where $\alpha>\kappa>d/2$, $s_j$ is the number of different irrational elements in $j$-th dimension of $Y_D^d$ and $L_j\in \calT_j(Y^d_{D})$ is the corresponding best simultaneous approximation sequence.

\end{theorem}

	\subsection{Discussion}
	\label{subsec:dis}
	
	\subsubsection{On the error bound}
	From the above analysis, we can make some discussion on 
	the approximation error.
	
	(i) When the quasiperiodic function $f(\bx)$ is known, we can use the equation \eqref{eq:procedure2} to directly obtain the approximate error $\|f_{p}-f\|_{\infty}$. 
	
	(ii) When the Fourier exponents of $f(\bx)$ and the periodic approximation function $f_{p}(\bx)$ are given, we can calculate an error bound $\varepsilon_1$ defined in \eqref{eq:summary}. Moreover, by solving \eqref{eq:procedure2}, we can obtain the Fourier coefficient vector $\by$.
	
	(iii) When the quasiperiodic function $f(\bx)$ is unknown, we can use \Cref{thm:bound of tf-f} to obtain an upper bound of $\|f_{p}-f\|_{\infty}$. 
	Moreover, the error bounds of $\|f_{p}-\calP_Nf\|_{\infty}$ have a relationship
	\begin{align*}
		\|f_{p}-\calP_Nf\|_{\infty} < \varepsilon_1<\varepsilon_2.
	\end{align*}

\subsubsection{On the best approximation rate}

Although directly computing the best simultaneous approximation sequence $\mathcal{T}$ can be challenging, we can still discuss its growth rate. 
The sequence $\mathcal{T}$ increases at a rate of at least \cite{Lagarias1982Best}
\begin{align*}
\liminf\limits_{k\rightarrow \infty}(t_k)^{1/k}\geq 1+\frac{1}{2^{s+1}}.
\end{align*}
The sequence $\mathcal{T}$ grows at a rate of at most \cite{Chevallier2001Meilleures}
\begin{align*}
\limsup\limits_{k\rightarrow \infty}\frac{1}{k}\ln t_k\leq C,
\end{align*}
for almost all $\bu=(u_{\ell})^{D}_{\ell=1}\in\bbR^{D}$ and $C$ is a contant. Here, ``almost all" refers to the Lebesgue measure on $\bbR^{D}$.

\section{Some examples}
\label{sec:NR}

This section offers two examples for $d=1$ and $d=3$
to support our theoretical results.
These examples involve only finite trigonometric summations, so there is no truncation error. Denote $\varepsilon_0=\|f_{p}-\calP_Nf\|_{\infty}$
and $e(L_j)=\sum_{\ell=1}^D \vert L_j\lambda_{\ell,j}-[L_j\lambda_{\ell,j}]\vert$, $j=1,2,\cdots,d$.
The $d$-dimensional quasiperiodic function $f(\bx)$ has the following expansion
$
f(\bx)=\sum\limits_{\ell=1}^{D} a_\ell e^{i2\pi (\bL\circ \blam_{\ell})\cdot \bx/\bL}.
$
The corresponding periodic approximation function $f_{p}(\bx)$ is given in \eqref{eq:periodic-trunction-n}
with a fundamental domain $[0,L_1)\times [0,L_2)\times\cdots \times [0,L_d)$.
We will show the rational approximation error $\varepsilon_0$ and two theoretical upper error bounds $\varepsilon_1$, $\varepsilon_2$.
To derive a more accurate upper bound $\varepsilon_2$, $x_1$ in \eqref{eq:summary} is calculated by
\begin{align*}
x_1=\max_{\zeta+1\leq s\leq D}\left\{
(\|\bv_{s}-\bh_{s}\|_{\ellinf}+\eta)^{2\eta (d-r_s)}
\bigg(\frac{1+\|\bv_{s}-\bh_{s}\|^{2}_{\ellinf}}
{1-\|\bv_{s}-\bh_{s}\|^{2}_{\ellinf}}\bigg)^{d-r_s}\right\}.
\end{align*}
	
\Cref{M-nonsingular} gives a sufficient condition that $M$ is invertible. In the following examples, this condition
can be weakened as
\begin{align}
\frac{L_{min}C_a}{(2N)^{2+\tau}}-\frac{1}{2}-\eta
> \max\Big\{1,~\Big[\frac{(\eta !)^2}{\pi}\Big]^{\frac{1}{2\eta+1}}\Big\},~~
G_{min}-2L_{max}\|\bP\|_1N - \frac{1}{2}>\eta.
\label{eq:LN-condition}
\end{align}
Let $\eta=1$, then
$[(\eta !)^2/ \pi]^{\frac{1}{2\eta+1}}=(1/\pi)^{\frac{1}{3}}<1$.
Inequality \eqref{eq:LN-condition} becomes
\begin{align*}
L_{min} > \frac{5}{2}\cdot\frac{(2N)^{2+\tau}}{C_a} ,~~
G_{min}>2L_{max}\|\bP\|_1N+\frac{3}{2}.
\end{align*}

\begin{example}
\label{EG:one}
Consider a one-dimensional quasiperiodic function $f(x)$
with four Fourier exponents
$(\lambda_1,~\lambda_2,~\lambda_3,~\lambda_4)=(1,~\sqrt{2},~2+\sqrt{2},~1+2\sqrt{2}).$
The corresponding Fourier coefficients are
$a_1=0.02-0.2i,~a_2=0.1,~a_3=0.03+0.1i,~a_4=0.02.$
\end{example}
In this example, the projection matrix is 
$\bP=(1~~\sqrt{2}).$
The reciprocal lattice vectors are
\begin{align*}
\begin{pmatrix}
\bk_1&\bk_2&\bk_3&\bk_4
\end{pmatrix}
=\begin{pmatrix}
1 & 0 & 2 & 1\\
0 & 1 & 1 & 2
\end{pmatrix},
\end{align*}
with $N=2$. Let Diophantine parameters $C_a=2$ and $\tau=0.2$. We can verify that $\lambda_1,~\lambda_2,~\lambda_3,~\lambda_4$ satisfy the Diophantine condition and
$\zeta=1,~I_{in}(\lambda_j)=\emptyset~ (j=2,3,4),~d_M=0,~\|\bP\|_1=\sqrt{2}$. Therefore, $L>20$ and $G>4\sqrt{2}L+3/2$.
Here, we choose $G=10L$.

We can obtain the periodic approximation function $f_{p}(x)$ from the equation \eqref{eq:procedure2}. For example, when $L=13860$, Fourier exponents of $f_{p}(x)$ are 
\begin{align*}
h_1 = 13860,~h_2=19601,~h_3=53062,~h_4=47321.
\end{align*}
The corresponding Fourier coefficient vector $\by_p$ is
{\small{
\begin{align*}
\by_p=\begin{pmatrix}
0.0200-0.2000i,0.1000-(8.0139 e-07)i,0.0300+0.1000i,0.0200-(1.6028e-07) i
\end{pmatrix}^{T}.
\end{align*}
}}

\Cref{fig:DV1} illustrates that as $L$ increases, $e(L)$ decreases, but the decrease is not uniform. To compute the error results, we present the first eight terms of the optimal approximation sequence $\mathcal{T}$, which correspond to the first column in \Cref{tab:1DPAM}. 
\begin{figure}[!htbp]
\centering
\includegraphics[width=4.6in]{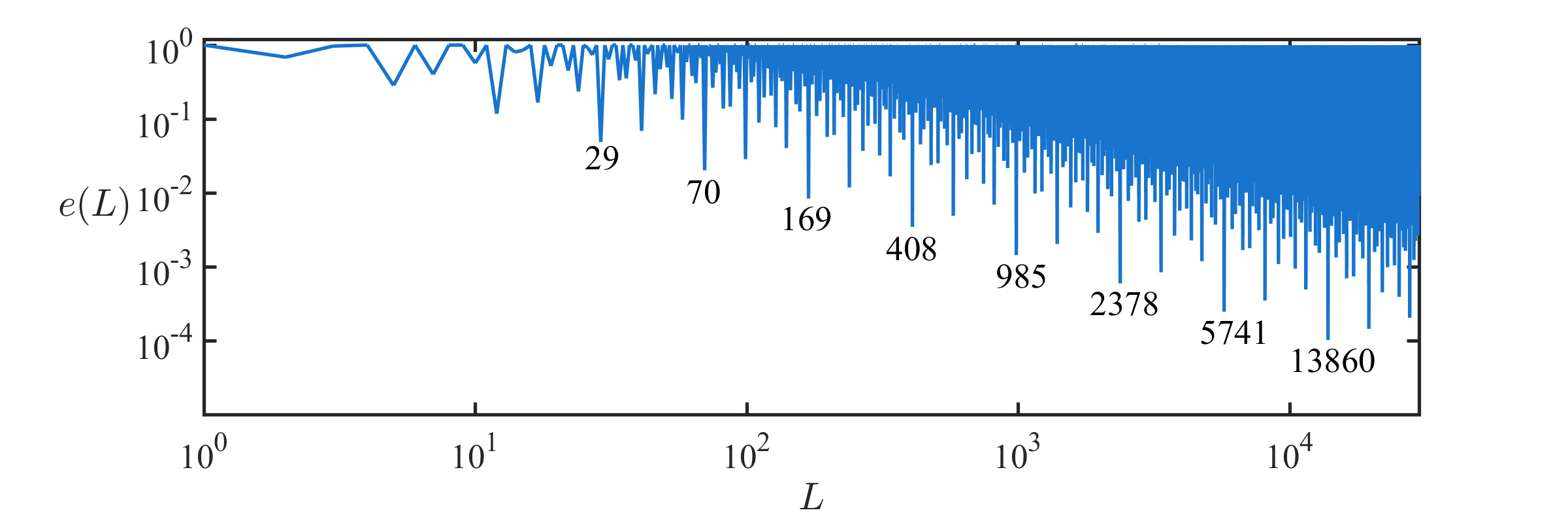}
\caption{The change of Diophantine approximation error $e(L)$ with an increase of $L$. The sequence in the figure corresponds to the best simultaneous approximation sequence.}\label{fig:DV1}
\end{figure}

	
\begin{table}[!hptb]
\centering
\footnotesize{
\caption{Error results of the one-dimensional case given in \Cref{EG:one}.}\label{tab:1DPAM}
\renewcommand\arraystretch{1.8}
\setlength{\tabcolsep}{3.2mm}
{\begin{tabular}{|c|c|c|c|c|}\hline
$L$    &$\|\DbV\|_e$ & $\varepsilon_0$ &$\varepsilon_1$    & $\varepsilon_2$     \\ \hline
29      &4.8773e-02    & 2.7689e-02              &9.2404e-02        &3.2843e-01          \\ \hline
70      &2.0203e-02    & 1.1549e-02               &3.8271e-02       &1.2979e-01          \\ \hline
169     &8.3682e-03    & 4.7935e-03               &1.5852e-02        &5.3200e-02         \\ \hline
408     &3.4662e-03    & 1.9877e-03               &6.5662e-03       &2.1948e-02          \\ \hline
985    &1.4357e-03    & 8.2512e-04             &2.7198e-03       &9.0765e-03          \\ \hline
2378    &5.9471e-04   &3.4217e-04           &1.1266e-03       &3.7571e-03        \\ \hline
5741    &2.4634e-04  &1.4180e-04            &4.6665e-04        &1.5558e-03         \\ \hline
13860    &1.0201e-04 &5.8747e-05            &1.9329e-04       &6.4436e-04          \\ \hline
\end{tabular}}
}
\end{table}

\Cref{tab:1DPAM} presents the rational approximation error $\varepsilon_0$ and the corresponding theoretical upper bounds $\varepsilon_1$ and $\varepsilon_2$. The relationship of $\varepsilon_0<\varepsilon_1<\varepsilon_2$ is consistent with the discussion in \Cref{subsec:dis}.

\Cref{fig:errorcompare} displays the reduction rates of these four errors when $L$ is selected as the best approximation sequence. It is evident that all four errors decrease at the rate of $O(L^{-1})$. Furthermore, we observe that the error $\varepsilon_0$ depends not only on the error $\Vert \DbV \Vert_e $ between Fourier exponents but also on the error $ \Vert \Delta \by \Vert$ between Fourier coefficients.

\begin{figure}[!htbp]
\centering
\includegraphics[width=2.4in]{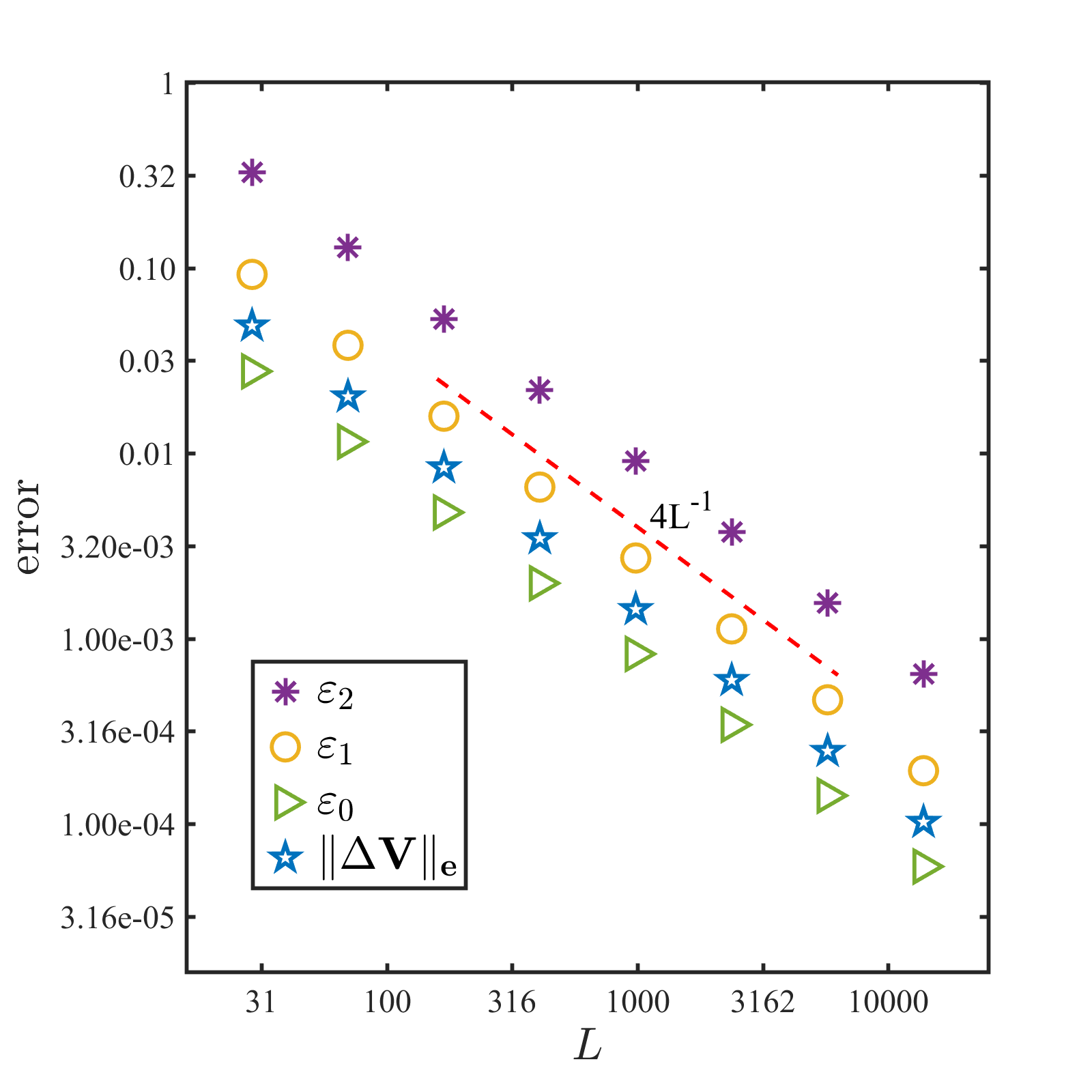}
\caption{In \Cref{EG:one}, the relationship between errors $\Vert \DbV\Vert_e$, $\varepsilon_0$, $\varepsilon_1$, $\varepsilon_2$ and $L$.}\label{fig:errorcompare}
\end{figure}

\begin{example}
\label{EG:three}
We consider two three-dimensional quasiperiodic functions.

Case (i). The Fourier exponents are
\begin{align*}
\bLam_1
=\begin{pmatrix}
1  & \sqrt{3}/2\\
0  & \sqrt{2}/2\\
\sqrt{3}/2 & 0
\end{pmatrix}.
\end{align*}
The corresponding Fourier coefficients are $a_1=0.2+0.1 i,~a_2=0.1 +0.2 i.$
Here, $d_M=0$.

Case (ii). The Fourier exponents are
\begin{align*}
\bLam_2
=\begin{pmatrix}
1 &  0         & \sqrt{5}/4\\
0 & \sqrt{2}/2         & 0\\
0 &\sqrt{3}/2 &\sqrt{3}/2
\end{pmatrix}.
\end{align*}
The corresponding Fourier coefficients are
$a_1=0.2+0.1 i,~a_2=0.1 +0.2 i,~a_3=0.02 -0.02 i.$
Here, $d_M=1$.
\end{example}

For Case (i), the projection matrix is
\begin{align*}
& \bP_1=\begin{pmatrix}
1& 0 & 0  &\sqrt{3}/2\\
0 & \sqrt{2}/2 & 0  &0\\
0 & 0        & \sqrt{3}/2   &0
\end{pmatrix},
\end{align*}
and $N=1$. The reciprocal lattice vectors are $\bk_1 = (1, 0, 1, 0)^T$ and $\bk_2 = (0, 1, 0, 1)^T$.
When Diophantine parameters $C_a=2$ and $\tau=0.1$,
$\zeta=0$, $I_{in}(\blam_j)\neq \emptyset,~(j=1,2)$, $d_m=1$, $\|\bP_1\|_1=1.$
Then $L_{min}>5$ and $G_{min}>2L_{max}+3/2$. Here, we choose $G_{min}=2L_{max}+10$.

It is evident that the second dimension is associated with $\sqrt{2}$, while the first and third dimensions are related to $\sqrt{3}$. Consequently, \Cref{fig:casei-subfig:a}-\Cref{fig:casei-subfig:b} show the relationship between $e(L_j)$ and $L_j$ with $j=1,2$, respectively. Note that $L_1=L_3$.

\begin{figure}[!htbp]
\centering
\subfigure[The relationship between $e(L_1)$ and $L_1$.]{
\label{fig:casei-subfig:a}
\includegraphics[width=4.6in]{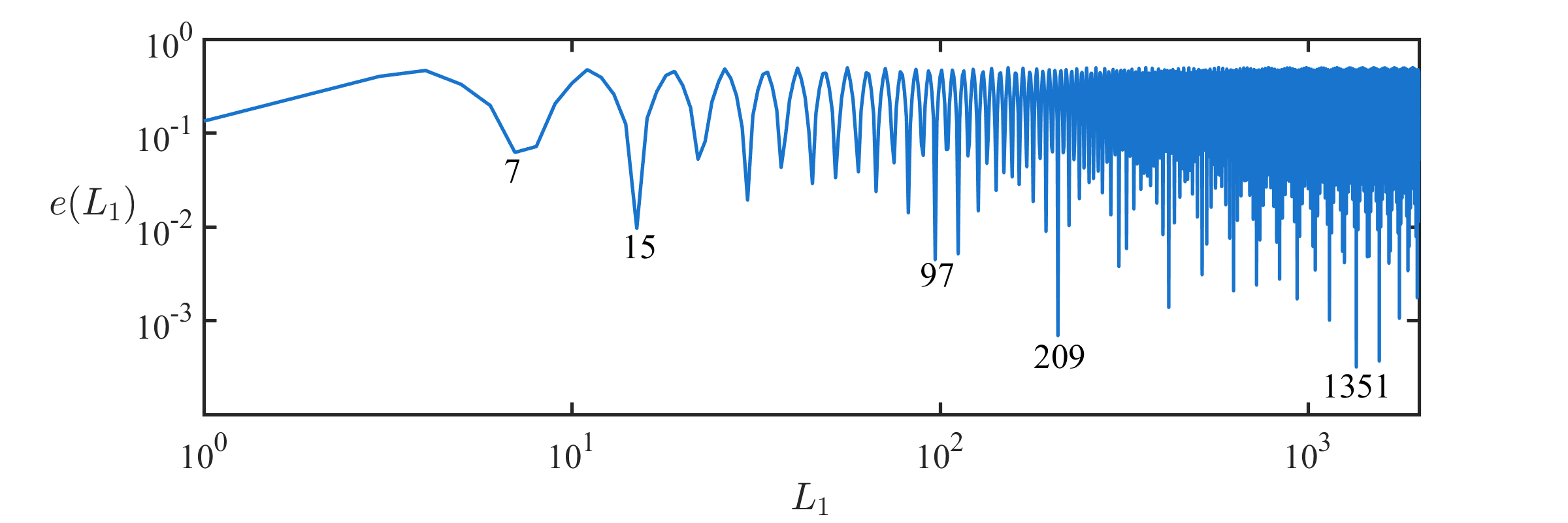}}\\
\subfigure[The relationship between $e(L_2)$ and $L_2$.]{
\label{fig:casei-subfig:b}
\includegraphics[width=4.6in]{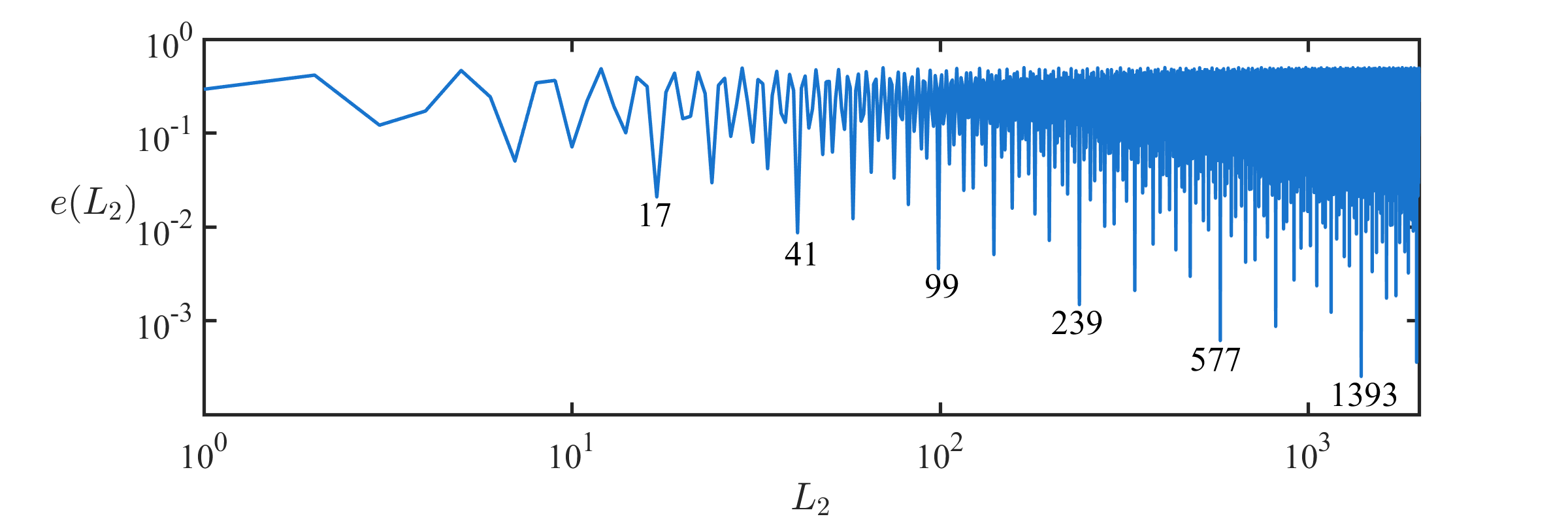}}
\caption{In Case (i) of \Cref{EG:three}, when the projection matrix is $\bP_1$, Diophantine approximation error $e(L_j)$ changes with increasing $L_j$ and $j=1,2$. The sequence in the subfigure corresponds to the best simultaneous approximation sequence in the corresponding dimension.}\label{fig:casei-eL}
\end{figure}

For Case (ii), the projection matrix is
\begin{align*}
& \bP_2=\begin{pmatrix}
1 & 0        & 0    &\sqrt{5}/4\\
0 & \sqrt{2}/2 & 0  &0\\
0 & 0        &\sqrt{3}/2   &0
\end{pmatrix}.
\end{align*}
The reciprocal lattice vectors are
$\bk_1 = (1, 0, 0, 0)^T$, $\bk_2 = (0, 1, 1, 0)^T$ and $\bk_3 = (0, 0, 1, 1)^T$ where $N=1$. When Diophantine parameters $C_a=2$ and $\tau=0.2$,
$\zeta=1,~I_{in}(\blam_j)\neq \emptyset,~(j=2,3),~ d_m=1,~\|\bP_2\|_1=1.$
Then $L_{min}>10$ and $G_{min}>2L_{max}+3/2$. Here, we choose $G_{min}=2L_{max}+10$.

Similarly, we show the relationship between $e(L_j)$ and $L_j$ as well as the best approximation sequence for each dimension. The second and third dimensions are presented in \Cref{fig:casei-eL}. \Cref{fig:caseii-eL} illustrates this relationship for the first dimension in Case (ii).

\begin{figure}[!htbp]
\centering
\includegraphics[width=4.6in]{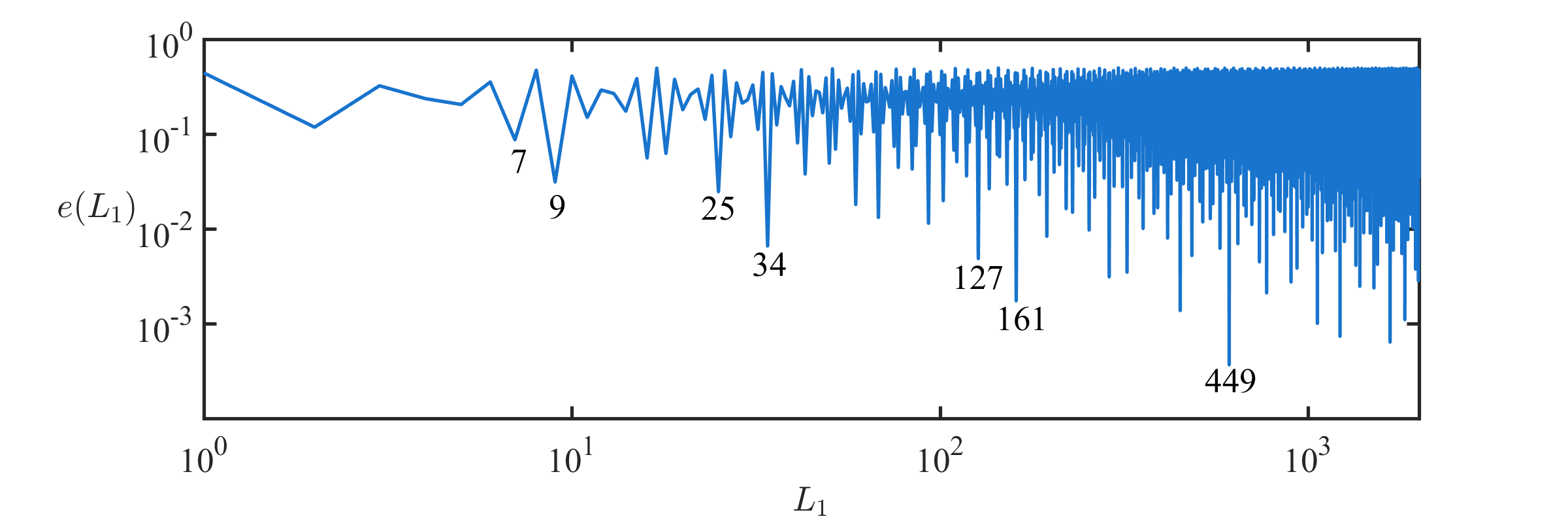}
\caption{In Case (ii), the change of Diophantine approximation error $e(L_1)$ with an increase of $L_1$.}\label{fig:caseii-eL}
\end{figure}

For convenience, let $G_1=G_2=G_3=G_{min}$ in these two cases.
The errors of three-dimensional quasiperiodic functions are presented in \Cref{tab:3}. The table compares three error bounds and clearly demonstrates the consistency of our theoretical findings.
Since the degrees of freedom of Case (i) and Case (ii) are given by $D=G_{min}^3$, the computational cost becomes significantly high as the area of $\bL$ increases. Therefore, we restrict the calculation area to $(0,209]\times(0,239]\times(0,209]$ and $(0,127]\times(0,99]\times(0,209]$ in \Cref{tab:3}, respectively.


\begin{table}[!hptb]
\vspace{-0.4cm}
\centering
\footnotesize{
\caption{Error results of two three-dimensional cases given in \cref{EG:three}.}\label{tab:3}
\renewcommand\arraystretch{1.8}
\setlength{\tabcolsep}{1.4mm}
{\begin{tabular}{|c|c|c|c|c|c|}\hline
The Fourier exponents  &$(L_1,L_2,L_3)$  &$\|\DbV\|_e$& $\varepsilon_0$ &$\varepsilon_1$   & $\varepsilon_2$   \\ \hline
\multirow{4}*{$\bLam_1$ } &(7,17,7)  &1.4517e-01 & 2.9179e-01 &3.0510e-01       &7.3318e-01   \\ \cline{2-6}
&(15,41,15)        &2.7860e-02    & 5.7767e-02               &5.8709e-02       &1.2095e-01          \\ \cline{2-6}
&(97,99,97)         &1.2500e-02       & 2.6158e-02     &  2.6342e-02       &5.3508e-02          \\ \cline{2-6}
&(209,239,209)         &2.8605e-03       &6.0135e-03   &6.0284e-03           &1.2147e-02  \\ \hline
\multirow{3}*{$\bLam_2$ } 
&(25,41,15)    &5.2435e-02    & 4.6980e-02     & 1.1050e-01      &3.1908e-01   \\ \cline{2-6}
&(34,99,97)   &1.9077e-02    & 1.9775e-02     &4.0205e-02       &1.0976e-01    \\ \cline{2-6}
&(127,99,209)   &9.7943e-03  & 6.0208e-03  & 1.9171e-02   &5.6053e-02 \\ \hline
\end{tabular}}
}
\end{table}

\section{Conclusion}
\label{sec:comments}
This paper presents a comprehensive theoretical error analysis of approximating an arbitrary-dimensional quasiperiodic function with a periodic function. The approximation error of this problem includes two parts: rational approximation error and truncation error. If the quasiperiodic function exhibits certain regularity, the rational approximation error dominates the approximation error. Meanwhile, we investigate the approximation rates of both the rational approximation error and the best periodic approximation sequence. Finally, we further verify the correctness of the theoretical analysis by several examples.

There are still many problems worth studying, including applying the PAM to solve the quasiperiodic differential equations/operators and providing the corresponding mathematical analysis based on these results presented here. Furthermore, we will develop the new method to analyze the approximation error of quasiperiodic functions with non-Diophantine frequencies by periodic functions.


\end{document}